\documentclass[reqno]{amsart}

\usepackage[a4paper, margin=1in]{geometry}

\usepackage {color}
\usepackage{amsmath,amsthm,amssymb}
\usepackage{epsfig}
\usepackage{graphicx}
\usepackage{amssymb}
\usepackage{latexsym}
\usepackage{pdfpages}
\usepackage{graphicx}
\usepackage{subcaption}
\usepackage{enumitem}

\usepackage{ulem} %para poder tachar: \sout{loquequierotachar}

\usepackage{ifpdf}
\newcommand{\res}{\!\!\mathop{\hbox{
                                \vrule height 7pt width .5pt depth 0pt
                                \vrule height .5pt width 6pt depth 0pt}}
                                \nolimits}

\def\z{{\bf z}}

 \usepackage[usenames,dvipsnames]{pstricks}
 \usepackage{epsfig}
 \usepackage{pst-grad} % For gradients
 \usepackage{pst-plot} % For axes

\usepackage{color}

\newtheorem{theorem}{Theorem}[section]
\newtheorem{lemma}[theorem]{Lemma}
\newtheorem{definition}[theorem]{Definition}
\newtheorem{proposition}[theorem]{Proposition}
\newtheorem{corollary}[theorem]{Corollary}
\newtheorem{remark}[theorem]{Remark}
\newtheorem{example}[theorem]{Example}
\newtheorem*{theorem*}{\it Theorem}

\def\vint_#1{\mathchoice%
          {\mathop{\kern 0.2em\vrule width 0.6em height 0.69678ex depth -0.58065ex
                  \kern -0.8em \intop}\nolimits_{\kern -0.4em#1}}%
          {\mathop{\kern 0.1em\vrule width 0.5em height 0.69678ex depth -0.60387ex
                  \kern -0.6em \intop}\nolimits_{#1}}%
          {\mathop{\kern 0.1em\vrule width 0.5em height 0.69678ex
              depth -0.60387ex
                  \kern -0.6em \intop}\nolimits_{#1}}%
          {\mathop{\kern 0.1em\vrule width 0.5em height 0.69678ex depth -0.60387ex
                  \kern -0.6em \intop}\nolimits_{#1}}}
\def\vintslides_#1{\mathchoice%
          {\mathop{\kern 0.1em\vrule width 0.5em height 0.697ex depth -0.581ex
                  \kern -0.6em \intop}\nolimits_{\kern -0.4em#1}}%
          {\mathop{\kern 0.1em\vrule width 0.3em height 0.697ex depth -0.604ex
                  \kern -0.4em \intop}\nolimits_{#1}}%
          {\mathop{\kern 0.1em\vrule width 0.3em height 0.697ex depth -0.604ex
                  \kern -0.4em \intop}\nolimits_{#1}}%
          {\mathop{\kern 0.1em\vrule width 0.3em height 0.697ex depth -0.604ex
                  \kern -0.4em \intop}\nolimits_{#1}}}

\def\R{\mathbb R}
\def\N{\mathbb N}
\def\Z{\mathbb Z}
\def\g{\hbox{\bf g}}

\numberwithin{equation}{section}

\def\1{\raisebox{2pt}{\rm{$\chi$}}}

\def\g{{\bf g}}

\definecolor{violet(ryb)}{rgb}{0.53, 0.0, 0.69}

\definecolor{olive(ryb)}{rgb}{0.42, 0.40, 0.18}
\usepackage{mathtools}
\mathtoolsset{showonlyrefs}
\begin{document}

\title[$(BV,L^p)$-decomposition  in  MRWS]{\bf $(BV,L^p)$-decomposition, $p=1,2$, of functions in  Metric Random Walk Spaces}

\author[J. M. Maz\'on, M. Solera, J. Toledo]{Jos\'e M. Maz\'on, Marcos Solera and Juli\'{a}n Toledo}

\address{J. M. Maz\'{o}n: Departamento de An\'{a}lisis Matem\'{a}tico,
Universitat de Val\`encia, Dr. Moliner 50, 46100 Burjassot, Spain.
%\hfill\break\indent
 {\tt mazon@uv.es }}
\address{ M. Solera: Departament d'An\`{a}lisi Matem\`atica,
Universitat de Val\`encia, Dr. Moliner 50, 46100 Burjassot, Spain.
%\hfill\break\indent
 {\tt marcos.solera@uv.es }}
\address{J. Toledo: Departament d'An\`{a}lisi Matem\`atica,
Universitat de Val\`encia, Dr. Moliner 50, 46100 Burjassot, Spain.
%\hfill\break\indent
 {\tt toledojj@uv.es }}

%\hfill\break\indent


\setcounter{tocdepth}{1}

\date{\today}

\begin{abstract}
In this paper we study the $(BV,L^p)$-decomposition, $p=1,2$, of functions in metric random walk spaces, a general workspace that includes weighted graphs and nonlocal models used  in image processing. We obtain the Euler-Lagrange equations  of the corresponding variational problems  and their gradient flows.  In the case $p=1$ we also study the associated geometric problem and the thresholding parameters.

\end{abstract}

\maketitle

{ \renewcommand\contentsname{Contents }
\setcounter{tocdepth}{3}
\tableofcontents
 }

\section{Introduction and Preliminaries}

A  metric random walk space $[X,d,m]$ is a metric space  $(X,d)$ together with a family  $m = (m_x)_{x \in X}$ of probability measures that  encode the jumps of a Markov chain. Important examples of metric random walk spaces are: locally finite weighted connected graphs, finite Markov chains and $[\R^N, d, m^J)$, with $d$ the Euclidean distance and
$$m^J_x(A) :=  \int_A J(x - y) d\mathcal{L}^N(y) \quad \hbox{for every Borel set } A \subset  \R^N ,$$
where $J:\R^N\to[0,+\infty[$ is a measurable, nonnegative and radially symmetric
function with mass equal to~$1$.
Furthermore, given a metric measure space $(X,d, \mu)$ we can obtain a metric random walk space $[X, d, m^{\mu,\epsilon}]$ called the {\it $\epsilon$-step random walk associated to $\mu$}, where
 $$m^{\mu,\epsilon}_x:= \frac{\mu \res B(x, \epsilon)}{\mu(B(x, \epsilon))}.$$

 Given a noisy  image $f:\Omega \rightarrow \R$ on a  rectangle $\Omega$  in $\R^2$, let us denote by $u : \Omega \rightarrow \R$  the desired cleaned image, which is related to the original one by $$f = u + n,$$
when  $n$  is an additive  noise. The problem of recovering $u$ from $f$ is ill-posed (see \cite{ACMBook}).
To handle this problem, Rudin, Osher and Fatemi (\cite{ROF}) proposed to solve the following constrained  minimization problem
\begin{equation}\label{ROFmod1}
\begin{array}{ll} \hbox{Minimize} \ \displaystyle\int_\Omega \vert Du \vert \quad \hbox{with} \ \ \displaystyle\int_\Omega u = \int_\Omega f, \ \  \int_\Omega \vert u - f\vert^2 = \sigma^2. \end{array}
\end{equation}
The first constraint corresponds to the assumption that the noise has zero mean, and the second that its standard deviation is $\sigma$. Problem \eqref{ROFmod1} is naturally linked to the following unconstrained problem (ROF-model):
\begin{equation}\label{ROFmod}
\min \left\{ \int_\Omega \vert Du \vert + \frac{\lambda}{2} \Vert u - f \Vert_2^2 \ : \ u \in BV(\Omega) \right\},
\end{equation}
for some Lagrange multiplier  $\lambda > 0$. Chambolle and Lions \cite{ChL} proved  an existence and uniqueness result for \eqref{ROFmod1},
as well as a proof of the link between \eqref{ROFmod1} and \eqref{ROFmod}. The constant $\lambda$ plays the role of a \lq\lq scale parameter". By tweaking $\lambda$, a user can select the level of detail desired in the reconstructed image.

The ROF-model leads to the $(BV,L^2)$-decomposition of $f$:
\begin{equation}\label{BV2decomp}
f = u_\lambda + v_\lambda , \quad [u_\lambda , v_\lambda] = \displaystyle{\rm argmin} \left\{ \int_\Omega \vert Du \vert + \frac{\lambda}{2} \Vert v \Vert_2^2  \ : \ f=u+v \right\} .
\end{equation}
 This decomposition   uses the $L^2$-fidelity term $\Vert f-u \Vert_2^2$. An alternative variational problem based on the $L^1$-fidelity term  $\Vert f -u\Vert_1$   was proposed by Alliney (\cite{Alliney}, \cite{Alliney1}) in one dimensional spaces and was studied extensively by Chan, Esedoglu and Nikolova (\cite{ChanEsedoglu}, \cite{ChanEsedogluNikolova}):
\begin{equation}\label{BV1decomp}
f = u_\lambda + v_\lambda , \quad [u_\lambda , v_\lambda] = \displaystyle{\rm argmin} \left\{ \int_\Omega \vert Du \vert + \lambda \Vert v \Vert_1  \ : \ f=u+v \right\} .
\end{equation}
The resulting $(BV,L^1)$ minimization differs from the $(BV,L^2)$ minimization in several important
aspects which have attracted considerable attention in recent years  (see \cite{AT}, \cite{Darbon}, \cite{DAG}, \cite{GO1}, \cite{YGO} and the references therein). Let us point out that the $(BV,L^1)$ minimization is contrast invariant (see \cite{ChanEsedoglu}), as opposed to
the $(BV,L^2)$ minimization.

 The use of  neighborhood filters by Buades, Coll and   Morel in  \cite{BCM}, that  was originally   proposed by P. Yaroslavsky \cite{Y1},  has led to an extensive literature of nonlocal models in image processing (see for instance \cite{BEM},  \cite{BEM1}, \cite{GO1}, \cite{KOJ}, \cite{LEL} and the references therein).  This nonlocal ROF-model, in a simplified version, has the form
 \begin{equation}\label{NonlocalROFmod}
\min \left\{ \int_{\Omega \times \Omega} J(x-y)   \vert u(x) - u(y) \vert dxdy + \frac{\lambda}{2} \Vert u - f \Vert_2^2 \ : \ u \in L^2(\Omega) \right\}.
\end{equation}

 On the other hand, an image can be seen as a weighted graph (see Example \ref{JJ} (3)) where the pixels are taken as the vertices, and the weights are related to the similarity between pixels. Depending on the problem there are different ways to define the weights, see for instance \cite{ELB}, \cite{HLE}, \cite{HLTE} and \cite{LEL}. The ROF-model in a weighted graph $G = (V(G), E(G), W(G))$ reads as follows:
  \begin{equation}\label{graphROFmod}
\min \left\{ \frac{1}{2}  \sum_{x \in V(G)} \sum_{y \in V(G)} \vert u(y) - u(x) \vert w_{xy}+ \frac{\lambda}{2} \sum_{x \in V(G)} \vert u(x) - f(x) \vert^2 \sum_{y \sim x} w_{xy}  : \ u \in L^2(G) \right\}.
\end{equation}

Problems  \eqref{NonlocalROFmod} and \eqref{graphROFmod} are particular cases of the following ROF-model in a metric random walk space $[X,d, m]$ with invariant  and reversible measure $\nu$:
 \begin{equation}\label{RWROFmod}
\min \left\{ \frac{1}{2} \int_{X}  \int_{X}  \vert u(y) - u(x) \vert dm_x(y) d\nu(x) + \frac{\lambda}{2} \int_X  \vert u(x) - f(x) \vert^2 d \nu(x)  : \ u \in L^2(X, \nu) \right\},
\end{equation}
 which is one of the motivations of this paper  and we call $m$-ROF-model.

 Another problem we are interested in is the $(BV,L^1)$ minimization in a metric random walk space $[X,d, m]$,  that reads as
  \begin{equation}\label{BVL1mod}
\min \left\{ \frac{1}{2} \int_{X}  \int_{X}  \vert u(y) - u(x) \vert dm_x(y) d\nu(x) + \lambda \int_X  \vert u(x) - f(x) \vert d \nu(x)  : \ u \in L^1(X, \nu) \right\},
\end{equation}
which has as a particular case the $(BV,L^1)$ minimization in graphs.

 The scale $\lambda$ in the $(BV,L^2)$-decomposition \eqref{BV2decomp} is viewed as a parameter that dictates the separation of the scale decomposition $f = u_\lambda + v_\lambda$. Following Meyer \cite{Meyer},   $u_\lambda$ extracts the edges of $f$ while $v_\lambda$ captures textures.

 In \cite{ROF}, to solve problem \eqref{ROFmod1},   the gradient descent method was used, which required to solve numerically the parabolic problem   \begin{equation}\label{ParabROFmod}
 \left\{ \begin{array}{lll} u_t = {\rm div} \left( \frac{ Du}{\vert Du \vert} \right)- \lambda (u - f) \quad &\hbox{in} \ (0, \infty) \times \Omega, \\ \\  \frac{ Du}{\vert Du \vert} \eta = 0 \quad &\hbox{on} \ (0, \infty) \times \partial \Omega, \\ \\ u(0, x) = v_0(x) \quad &\hbox{in} \ x \in \Omega,\end{array} \right.
 \end{equation}
 then the denoised version of $f$ is approached by the solution of~\eqref{ParabROFmod}  as $t$ increases.    The concept of solution for  which such problem  is well-possed  was given in \cite{ABCM1}.  We will see here that a non-local version of~\eqref{ParabROFmod} can be used to  approach the solutions of the ROF-problem in the workspace of metric random walk spaces (see Theorem~\ref{sab1001}).

Our aim is to  study the $(BV,L^p)$-decomposition, $p=1,2$, of functions in  metric random walk spaces,
developing a general theory that can be applied to weighted graphs and nonlocal models.

 \subsection{Metric Random Walk Spaces}

Let $(X,d)$ be a  Polish metric  space equipped with its Borel $\sigma$-algebra.
A {\it random walk} $m$ on $X$ is a family of probability measures $m_x$ on $X$, $x \in X$, satisfying the two technical conditions: (i) the measures $m_x$  depend measurably on the point  $x \in X$, i.e., for any Borel set $A\subset X$ and any Borel set $B\subset \R$, the set $\{ x \in X \ : \ m_x(A) \in B \}$ is Borel; (ii) each measure $m_x$ has finite first moment, i.e. for some (hence any) $z \in X$ and for any $x \in X$ one has $\int_X d(z,y) dm_x(y) < +\infty$ (see~\cite{O}).

A {\it metric random walk  space} (or {\it MRW space}) $[X,d,m]$  is a  Polish metric space $(X,d)$ together with a  random walk $m$.

Let $[X,d,m]$ be a metric random walk  space. A   Radon measure $\nu$ on $X$ is {\it invariant} with respect to the random walk $m=(m_x)$ if
$$d\nu(x)=\int_{y\in X}d\nu(y)dm_y(x),$$
that is,  for any $\nu$-measurable set $A\subset X$, it holds that $A$ is $m_x$-measurable  for $\nu$-almost all $x\in X$,  $\displaystyle x\mapsto  m_x(A)$ is $\nu$-measurable, and
$$\nu(A)=\int_X m_x(A)d\nu(x).$$
Consequently, if $\nu$ is an invariant measure with respect to $m$ and $f \in L^1(X, \nu)$, it holds that $f \in L^1(X, m_x)$ for $\nu$-a.e. $x \in X$, $\displaystyle x\mapsto \int_X f(y) d{m_x}(y)$ is $\nu$-measurable, and \label{paginv}
$$\int_X f(x) d\nu(x) = \int_X \left(\int_X f(y) d{m_x}(y) \right)d\nu(x).$$

The measure $\nu$ is said to be {\it reversible} for $m$ if moreover, the detailed balance condition $$dm_x(y)d\nu(x)  = dm_y(x)d\nu(y) $$ holds true. Under suitable assumptions on the metric random walk space $[X,d,m]$, such an invariant and reversible measure $\nu$ exists and is unique, as we will see below. Note that the reversibility condition implies the invariance condition.

 We will assume that the  metric measure space $(X,d)$ is $\sigma$-finite.

 \begin{example}\label{JJ}{\rm
  \begin{enumerate}
   \item \label{dom001}
Consider $(\R^N, d, \mathcal{L}^N)$, with $d$ the Euclidean distance and $\mathcal{L}^N$ the Lebesgue measure. Let  $J:\R^N\to[0,+\infty[$ be a measurable, nonnegative and radially symmetric
function  verifying $\int_{\R^N}J(z)dz=1$. In $(\R^N, d, \mathcal{L}^N)$ we have the following random walk, starting at $x$,
$$m^J_x(A) :=  \int_A J(x - y) d\mathcal{L}^N(y) \quad \hbox{for every Borel set } A \subset  \R^N.$$
Applying Fubini's Theorem it is easy to see that the Lebesgue measure $\mathcal{L}^N$ is an invariant and reversible measure for this random walk.

In the case that $\Omega \subset \R^N$ is a closed bounded set, if, for $x \in \Omega$, we define

$$m^{J,\Omega}_x(A):=\int_A J(x-y)dy+\left(\int_{\R^n\setminus \Omega}J(x-z)dz\right)\delta_x(A) \quad \forall A \subset  \Omega \ \hbox{Borel},$$
we have that each $m^{J,\Omega}_x$ is a probability measure in $\Omega$. Moreover, it is easy to see (\cite{MST0}) that
 $\nu = \mathcal{L}^N \res \Omega$ is an invariant and reversible measure for the random walk $m^{J,\Omega} = ( m^{J,\Omega}_x)_{x \in \Omega}$.

 \item  Let $K: X \times X \rightarrow \R$ be a Markov kernel on a countable space $X$, i.e.,
 $$K(x,y) \geq 0, \quad \forall x,y \in X, \quad \quad \sum_{y\in X} K(x,y) = 1 \quad \forall x \in X.$$
 Then, for $$m^K_x(A):= \sum_{y \in A} K(x,y),$$
 $[X, d, m^K]$ is a metric random walk space for  any  metric  $d$ on $X$. Basic Markov chain theory guarantees the existence of a unique  stationary probability measure (also called steady state) $\pi$ on $X$, that is,
 $$\sum_{x \in X} \pi(x) = 1 \quad \hbox{and} \quad \pi(y) = \sum_{x \in X} \pi(x) K(x,y) \quad \quad \forall y \in X.$$
 We say that $\pi$ is reversible for $K$ if the following detailed balance equation
 $$K(x,y) \pi(x) = K(y,x) \pi(y)$$ holds true for $x, y \in X$.

 \item \label{dom003} Consider  a locally finite weighted discrete graph $G = (V(G), E(G))$, where each edge $(x,y) \in E(G)$ (we will write $x\sim y$ if $(x,y) \in E(G)$) has a positive weight $w_{xy} = w_{yx}$ assigned. Suppose further that $w_{xy} = 0$ if $(x,y) \not\in E(G)$. The graph is equipped with the standard shortest path graph distance $d_G$, that is, $d_G(x,y)$ is the  minimal number of
edges connecting $x$ and $y$.  We will assume that any two points are connected, i.e., that the graph is connected. For $x \in V(G)$ we define the weight   at the vertex $x$ as $$d_x:= \sum_{y\sim x} w_{xy} = \sum_{y\in V(G)} w_{xy},$$
and the neighbourhood $N_G(x) := \{ y \in V(G) \ : \ x\sim y\}$. By definition of locally finite graph, the sets $N_G(x)$ are finite. When $w_{x,y}=1$ for every $x,y\in V(G)$, $d_x$ coincides with the degree of the vertex $x$ in the graph, that is,  the number of edges containing vertex $x$.

  For each $x \in V(G)$  we define the following probability measure
\begin{equation}\label{discRW}m^G_x=  \frac{1}{d_x}\sum_{y \sim x} w_{xy}\,\delta_y.\\ \\
\end{equation}
 We have that $[V(G), d_G, (m^G_x)]$ is a metric random walk space. It is not difficult to see that the measure $\nu_G$ defined as
 $$\nu_G(A):= \sum_{x \in A} d_x,  \quad A \subset V(G),$$
is an invariant and  reversible measure for this random walk.

Given a locally finite weighted discrete graph $G = (V(G), E(G))$, there is a natural definition of a Markov chain
on the vertices. We define the Markov kernell $K_G: V(G)\times V(G) \rightarrow \R$ as
$$K_G(x,y):= \frac{1}{d_x}  w_{xy}.$$
We have that $m^G$ and $m^{K_G}$ define the same random walk.  If $\nu_G(V(G))$ is finite,   the unique stationary and reversible probability measure is given by
$$\pi_G(\{x\}):= \frac{1}{\nu_G(V(G))} \sum_{z \in V(G)} w_{xz}. $$

 In Machine Learning Theory (\cite{G-TS}, \cite{G-TSBLBr}) an example of a weighted discrete graph is a point cloud in $\R^N$,  $V=\{x_1, \ldots , x_n \}$,  with edge weights $w_{x_i,x_j}$ given by
$$w_{x_i,x_j} := \eta(\vert x_i - x_j\vert), \quad 1 \leq i,j \leq n,$$
where the kernel  $ \eta : [0, \infty) \rightarrow [0, \infty)$ is a radial profile satisfying
\begin{itemize}
\item[(i)] \ $ \eta(0) >0$, and $\eta$ is continuous at $0$,
\item[(ii)]  \ $\eta$ is non-decreasing,
\item[(iii)] \ and the integral $\int_0^\infty  \eta(r) r^N dr$ is finite.
\end{itemize}

\item \label{dom006} From a metric measure space $(X,d, \mu)$ we can obtain a metric random walk space, the so called {\it $\epsilon$-step random walk associated to $\mu$}, as follows. Assume that balls in $X$ have finite measure and that ${\rm Supp}(\mu) = X$. Given $\epsilon > 0$, the $\epsilon$-step random walk on $X$, starting at point~$x$, consists in randomly jumping in the ball of radius $\epsilon$ around $x$, with probability proportional to $\mu$; namely
 $$m^{\mu,\epsilon}_x:= \frac{\mu \res B(x, \epsilon)}{\mu(B(x, \epsilon))}.$$
Note that $\mu$ is an invariant and reversible measure for the metric random walk space $[X, d, m^{\mu,\epsilon}]$.

\item \label{dom00606} Given a  metric random walk  space $[X,d,m]$ with invariant and reversible measure $\nu$, and given a $\nu$-measurable set $\Omega \subset X$ with $\nu(\Omega) > 0$, if we define, for $x\in\Omega$,
$$m^{\Omega}_x(A):=\int_A d m_x(y)+\left(\int_{X\setminus \Omega}d m_x(y)\right)\delta_x(A) \quad \hbox{ for every Borel set } A \subset  \Omega  ,
$$
 we have that $[\Omega,d,m^{\Omega}]$ is a metric random walk space and it easy to see that $\nu \res \Omega$ is  reversible for $m^{\Omega}$.
 \end{enumerate}
}
\end{example}

From now on   we will assume that $[X,d,m]$ is a metric random walk space with a   reversible measure $\nu$ for $m$.

\subsection{$m$-Perimeter}
We introduce the {\it $m$-interaction} between two $\nu$-measurable subsets $A$ and $B$ of $X$  as
\begin{equation}\label{nlinterdos} L_m(A,B):= \int_A \int_B dm_x(y) d\nu(x).
 \end{equation}
Whenever  $L_m(A,B) < \infty$, by the reversibility assumption on $\nu$ with respect to $m$, we have
 $$L_m(A,B)=L_m(B,A).$$

 We define the concept of $m$-perimeter of a $\nu$-measurable subset $E \subset X$ as
 $$P_m(E)=L_m(E,X\setminus E) = \int_E \int_{X\setminus E} dm_x(y) d\nu(x).$$
 It is easy to see that
\begin{equation}\label{secondf}P_m(E) = \frac{1}{2} \int_{X}  \int_{X}  \vert \1_{E}(y) - \1_{E}(x) \vert dm_x(y) d\nu(x).
\end{equation}
Moreover,  if $\nu(E)<+\infty$, we have
 \begin{equation}\label{secondf021}\displaystyle P_m(E)=\nu(E) -\int_E\int_E dm_x(y) d\nu(x).
\end{equation}

\begin{example}\label{dt}{\rm
\begin{enumerate}
\item \label{dt001} Let  $[\R^N, d, m^J]$ be   the metric random walk space given in Example   \ref{JJ} (\ref{dom001}) with invariant measure $\mathcal{L}^N$. Then,
    $$P_{m^J} (E) = \frac{1}{2} \int_{\R^N}  \int_{\R^N}  \vert \1_{E}(y) - \1_{E}(x) \vert J(x -y) dy dx,$$
    which coincides with the concept of $J$-perimeter introduced in \cite{MRT1} (see also \cite{MRTLibro}). On the other hand,
    $$P_{ m^{J,\Omega}} (E) = \frac{1}{2} \int_{\Omega}  \int_{\Omega}  \vert \1_{E}(y) - \1_{E}(x) \vert J(x -y) dy dx.$$
    Note that, in general, $P_{ m^{J,\Omega}} (E) \not= P_{m^J} (E).$

\item \label{dt003}

In the case of the metric random walk space $[V(G), d_G, m^G ]$ associated to a finite weighted discrete graph $G$, given $A, B \subset V(G)$, ${\rm Cut}(A,B)$ is defined as
$${\rm Cut}(A,B):= \sum_{x \in A, y \in B} w_{xy} = L_{m^G}(A,B),$$
and the perimeter of a set $E \subset V(G)$ is given by
$$\vert \partial E \vert := {\rm Cut}(E,E^c) = \sum_{x \in E, y \in V \setminus E} w_{xy}.$$
Consequently, we have that
\begin{equation}\label{perim}
\vert \partial E \vert = P_{m^G}(E) \quad \hbox{for all} \ E \subset V(G).
\end{equation}
\end{enumerate}
}

\end{example}

  Let us now remember some properties of the $m$-perimeter given in~\cite{MST1}.

\begin{proposition}[\cite{MST1}]\label{launion01}  1. Let $A,\ B \subset X$ be $\nu$-measurable sets with finite $m$-perimeter such that $\nu(A \cap B) = 0$. Then,
  $$P_m( A \cup B) = P_m( A) + P_m(B) - 2 L_m(A,B).$$

\noindent 2. Let $A,\ B,\ C$ be $\nu$-measurable sets in $X$ with pairwise $\nu$-null intersections. Then
 $$
 \displaystyle P_m( A \cup B\cup C)=P_m( A \cup B) +P_m(A\cup C) + P_m(B\cup C) - P_m( A) - P_m(B)-P_m(C) .
  $$
\end{proposition}

  \begin{proposition}[Submodularity]\label{Juli1sub}
Let $A$ and $B$ be $\nu$-measurable sets in $X$. Then
$$ P_m( A \cup B)+ P_m(A\cap B)= P_m(A) + P_m(B)- 2L_m(A\setminus B, B\setminus A) .$$
Consequently,
 $$ P_m( A \cup B)+ P_m(A\cap B)\le P_m(A) + P_m(B) .$$
\end{proposition}

\begin{proof}
By Proposition~\ref{launion01}\ {\it 2}, $$\begin{array}{l}
\displaystyle
P_m(A\cup B)=P_m\big((A\setminus B)\cup (B\setminus A)\cup (A\cap B)\big)\\ \\
\displaystyle
\phantom{P_m(A\cup B)}=P_m\big((A\setminus B)\cup (B\setminus A)\big)+P_m(A)+P_m(B)\\ \\
\displaystyle
\phantom{P_m(A\cup B)=}
-P_m(A\setminus B)-P_m(B\setminus A)-P_m(A\cap B).
\end{array}
$$
Hence,
$$
\begin{array}{l}
\displaystyle P_m(A\cup B)+P_m(A\cap B)=P_m(A)+P_m(B)\\ \\
\displaystyle
\phantom{P_m(A\cup B)+P_m(A\cap B)=} +P_m\big((A\setminus B)\cup (B\setminus A)\big)-P_m(A\setminus B)-P_m(B\setminus A).
\end{array}$$
Now, by Proposition~\ref{launion01}\ {\it 1.},
$$P_m\big((A\setminus B)\cup (B\setminus A)\big)-P_m(A\setminus B)-P_m(B\setminus A ) = - 2L_m(A\setminus B, B\setminus A).$$
\end{proof}

\subsection{$m$-Total Variation}
 Associated with the  random walk $m=(m_x)$ and the invariant measure $\nu$,  we define the space
 $$BV_m(X):= \left\{ u :X \to \R \ \hbox{ $\nu$-measurable} \, : \, \int_{X}  \int_{X}  \vert u(y) - u(x) \vert dm_x(y) d\nu(x) < \infty \right\}.$$
 We have that $L^1(X,\nu)\subset BV_m(X)$. The {\it $m$-total variation} of a function $u\in BV_m(X)$ is defined by
 $$TV_m(u):= \frac{1}{2} \int_{X}  \int_{X}  \vert u(y) - u(x) \vert dm_x(y) d\nu(x).$$
 Note that
 $$P_m(E) = TV_m(\1_E).$$

Recall the definition of the generalized product measure $\nu \otimes m_x$ (see, for instance, \cite{AFP}), which is defined as the measure in $X \times X$ given by
 $$ \nu \otimes m_x(U) := \int_X   \int_X \1_{U}(x,y) dm_x(y)   d\nu(x)\quad\hbox{for }  U\in \mathcal{B}(X\times X),$$  where one needs the map $x \mapsto m_x(E)$ to be $\nu$-measurable for any Borel set $E \in \mathcal{B}(X)$.
 It holds that $$\int_{X \times X} g  d(\nu \otimes m_x)   = \int_X   \int_X g(x,y) dm_x(y)  d\nu(x)$$
 for every  $g\in L^1(X\times X,\nu\otimes m_x)$. Therefore, we can write
$$TV_m(u)= \frac{1}{2} \int_{X}  \int_{X}  \vert u(y) - u(x) \vert d(\nu \otimes m_x)(x,y).$$

\begin{example}\label{exammplee}{\rm Let $[V(G), d_G, (m^G_x)]$  be the metric random walk space given in Example   \ref{JJ} (\ref{dom003}) with invariant measure $\nu_G$. Then
$$TV_{m^G} (u) = \frac{1}{2} \int_{V(G)}  \int_{V(G)}  \vert u(y) - u(x) \vert dm^G_x(y) d\nu_G(x) $$ $$= \frac{1}{2} \int_{V(G)} \frac{1}{d_x} \left(\sum_{y \in V(G)} \vert u(y) - u(x) \vert w_{xy}\right) d\nu_G(x)$$ $$=  \frac{1}{2} \sum_{x \in V(G)} d_x \left(\frac{1}{d_x} \sum_{y \in V(G)} \vert u(y) - u(x) \vert w_{xy}\right) = \frac{1}{2}  \sum_{x \in V(G)} \sum_{y \in V(G)} \vert u(y) - u(x) \vert w_{xy}.$$
 Note that $TV_{m^G} (u)$ coincides with the anisotropic total variation $TV^1_a(u)$ defined in \cite{GGOB}.
}
\end{example}

The next results are proved in \cite{MST1}.
\begin{lemma}[\cite{MST1}]\label{semicont} $TV_m$ is lower semi-continuous with respect to the the weak convergence in $L^2(X, \nu)$.
\end{lemma}
\begin{theorem}[\bf Coarea formula, \cite{MST1}]\label{1coarea1}
   For any $u \in L^1(X,\nu)$, let $E_t(u):= \{ x \in X \, : \, u(x) > t \}$. Then,
\begin{equation}\label{coaerea}
TV_m(u) = \int_{-\infty}^{+\infty} P_m(E_t(u))\, dt.
\end{equation}
\end{theorem}

\subsection{The $1$-Laplacian   in MRW Spaces}\label{1Lpalce}

Given a function $u : X \rightarrow \R$ we define its nonlocal gradient $\nabla u: X \times X \rightarrow \R$ as
$$\nabla u (x,y):= u(y) - u(x) \quad \forall \, x,y \in X.$$
For a function $\z : X \times X \rightarrow \R$, its {\it $m$-divergence} ${\rm div}_m \z : X \rightarrow \R$ is defined as
 $$({\rm div}_m \z)(x):= \frac12 \int_{X} (\z(x,y) - \z(y,x)) dm_x(y).$$

For $p \geq 1$, we define the space
$$X_m^p(X):= \left\{ \z \in L^\infty(X\times X, \nu \otimes m_x) \ : \ {\rm div}_m \z \in L^p(X,\nu) \right\}.$$

Let $u \in BV_m(X) \cap L^{p'}(X,\nu)$ and $\z \in X_m^p(X)$,  $1\le p\le \infty$, having in mind that $\nu$ is reversible, we have the following {\it Green's formula}
\begin{equation}\label{Green}
\int_{X} u(x) ({\rm div}_m \z)(x) dx = -\frac12 \int_{X \times X} \nabla u(x,y) \z(x,y)   d\nu\otimes dm_x.
\end{equation}

Let us denote by
$\hbox{sign}(r)$ the multivalued sign function
$${\rm sign}(u)(x):=  \left\{ \begin{array}{lll} 1 \quad \quad &\hbox{if} \ \  u(x) > 0, \\ -1 \quad \quad &\hbox{if} \ \ u(x) < 0, \\ \left[-1,1\right] \quad \quad &\hbox{if} \ \ \ u(x) = 0. \end{array}\right.$$

The functional $\mathcal{F}_m : L^2(X, \nu) \rightarrow ]-\infty, + \infty]$ defined by
$$\mathcal{F}_m(u):= \left\{ \begin{array}{ll} \displaystyle
TV_m(u)
 \quad &\hbox{if} \ u\in L^2(X,\nu)\cap BV_m(X), \\[10pt] + \infty \quad &\hbox{if } u\in L^2(X,\nu)\setminus BV_m(X), \end{array} \right.$$
 is convex and lower semi-continuous.
The following characterizations for its subdifferential  have been proved in~\cite{MST1}.

\begin{theorem}[\cite{MST1}]\label{chsubd} Let $u \in  L^2(X,\nu)$ and $v \in L^2(X,\nu)$. The following assertions are equivalent:

\item[ (i)] $v \in \partial \mathcal{F}_m (u)$;

\item[ (ii)] there   exists  $\z  \in X_m^2(X)$, $\Vert \z \Vert_\infty \leq 1$ such that
\begin{equation}\label{eqq2}
   v = - {\rm div}_m \z
\end{equation}
and
\begin{equation}\label{eqq1}
\int_{X} u(x) v(x) d\nu(x) = \mathcal{F}_m (u);
\end{equation}

\item[ (iii)] there exists $\z  \in X_m^2(X)$, $\Vert \z \Vert_\infty \leq 1$ such that \eqref{eqq2} holds and
\begin{equation}\label{eqq3}
\mathcal{F}_m (u) = \frac12\int_{X \times X} \nabla u(x,y) \z(x,y) d\nu\otimes dm_x;
\end{equation}

\item[ (iv)] there exists $\hbox{\bf g}\in L^\infty(X\times X, \nu \otimes m_x)$ antisymmetric with $\Vert \hbox{\bf g} \Vert_\infty \leq 1$
        such that
    \begin{equation}\label{1-lapla.var-ver}
-\int_{X}\g(x,y)\,dm_x(y)=  v(x) \quad \nu-\mbox{a.e }x\in X,
\end{equation}
         and
     \begin{equation}\label{1-lapla.sign}-\int_{X} \int_{X}\hbox{\bf g}(x,y)dm_x(y)\,u(x)d\nu(x)=\mathcal{F}_m(u).
     \end{equation}

     \item[ (v)] there exists $\hbox{\bf g}\in L^\infty(X\times X, \nu \otimes m_x)$ antisymmetric with $\Vert \hbox{\bf g} \Vert_\infty \leq 1$ verifying \eqref{1-lapla.var-ver} and
         \begin{equation}\label{1-lapla.sign2}\hbox{\bf g}(x,y) \in {\rm sign}(u(y) - u(x)) \quad (\nu \otimes m_x)-a.e. \ (x,y) \in X \times X.
     \end{equation}
\end{theorem}

 \begin{remark}{\rm
The next space was introduced in \cite{Meyer} in the continuous setting. Set
$$G_m(X):= \{ f \in L^2(X, \nu) \ : \ \exists \z \in X^2_m(X)  \ \hbox{such that} \ f = {\rm div}_m(\z)\}$$
and consider in $G_m(X)$ the norm
$$\Vert f \Vert_{m,*} := \inf\{\Vert \z \Vert_\infty \ : \ f = {\rm div}_m(\z)\}.$$
From the proof of~\cite[Theorem 3.3]{MST1} we have that
 \begin{equation}\label{meyer1}
 \Vert f \Vert_{m,*} = \sup \left\{\left\vert \int_X f(x) u(x) d\nu(x) \right\vert \ : u \in L^2(X, \nu) , \ TV_m(u) \leq 1\right\},
 \end{equation}
and
\begin{equation}\label{perfect1}
\displaystyle
\partial \mathcal{F}_m (u)
   = \left\{ v \in G_m(X)  \, : \, \Vert v \Vert_{m,*} \leq 1, \ \  \int_{X} u(x) v(x) d\nu(x) = \mathcal{F}_m(u)\right\}.
\end{equation}
In particular,
\begin{equation}\label{porfin1}\partial  \mathcal{F}_m (0) = \{v \in G_m(X) \, : \, \Vert v \Vert_{m,*} \leq 1 \}.\end{equation}
}
\end{remark}

\begin{definition}\label{el1laplac}{\rm
We define the multivalued operator $\Delta^m_1$ in $L^2(X,\nu)$ by $(u, v ) \in \Delta^m_1$ if, and only if, $-v \in \partial \mathcal{F}_m(u)$.
}
\end{definition}

  It is easy to see that $-\Delta^m_1$ is a completely accretive operator (see~\cite{BCr2}).

 Chang, in \cite{Chang1}, and  Hein and B\"uhler, in \cite{HB}, define  a similar operator in the particular case of finite graphs.

\medskip

 We finish this section by recalling some concepts   studied in~\cite{MST1}   that will be used later.
\begin{definition}\label{Ndefeigenpair}{\rm
A pair $(\lambda, u) \in \R \times L^2(X, \nu)$ is called an {\it $m$-eigenpair} of the $1$-Laplacian $\Delta^m_1$ on $X$ if $\Vert u \Vert_1 = 1$ and    there exists $\xi \in {\rm sign}(u)$   such that
\begin{equation}\label{deff1}
\lambda \, \xi \in    - \Delta^m_1 u.
\end{equation}
The constant $\lambda$ is called {\it $m$-eigenvalue} and the function $u$ is an {\it $m$-eigenfunction} associated to $\lambda$.
}
\end{definition}

Given a function $u : X \rightarrow \R$,   $\mu \in \R$ is a {\it median} of $u$ with respect to the measure $\nu$ if
$$\nu(\{ x \in X \ : \ u(x) < \mu \}) \leq \frac{1}{2} \nu(X), \quad \nu(\{ x \in X \ : \ u(x) > \mu \}) \leq \frac{1}{2} \nu(X).$$
We denote by ${\rm med}_\nu (u)$ the set of all medians of $u$. It is easy to see that

\begin{equation}\label{NNsigg1}{\rm argmin}\left\{\int_X|f-c| \ : \ c\in \R \right\}={\rm med}_\nu(f)\end{equation}
and
\begin{equation}\label{Nsigg1}
0 \in {\rm med}_\nu (u) \iff \exists \xi \in {\rm sign}(u) \ \hbox{such that} \ \int_X \xi(x) d \nu(x) = 0.
\end{equation}

Given a  set $\Omega \subset X$ with    $0 < \nu(\Omega) < \nu(X)$, we   denote
$$\lambda^m_\Omega:= \frac{P_m(\Omega)}{\nu(\Omega )},$$
and  we define the {\it $m$-Cheeger constant} of $\Omega$ by
\begin{equation}\label{cheeg1}h_1^m(\Omega) := \inf \left\{ \lambda^m_E \, : \, E \subset \Omega, \ E \ \hbox{$\nu$-measurable and }  \,   \nu( E)>0 \right\}.\end{equation}
A $\nu$-measurable set $E  \subset \Omega$ achieving the infimum in \eqref{cheeg1} is said to be an {\it $m$-Cheeger set} of $\Omega$. Furthermore, we say that $\Omega$ is {\it $m$-calibrable} if it is an $m$-Cheeger set of itself, that is, if $h_1^m(\Omega) = \lambda^m_\Omega.$

\medskip

 \begin{definition}\label{ergodef}{\rm
 A Borel set $B \subset X$ is said to be {\it invariant} with respect to the random walk $m$ if $m_x(B) = 1$ whenever $x\in B$.
 An invariant  probability measure~$\nu$ is said to be {\it ergodic} with respect to $m$ if $\nu(B) = 0$ or $\nu(B) = 1$ for every invariant set $B$ with respect to the random walk $m$.
}
\end{definition}

If $\nu$ is ergodic with respect to $m$ then (see~\cite{MST1})
 $$TV_m(u) = 0 \iff u \ \hbox{is $\nu$-a.e. equal to a constant}.$$

From now on we will assume that $\nu$ is finite and that $\frac{1}{\nu(X)}\nu$ is ergodic with respect to $m$. See~\cite{MST1} for characterizations of the ergodicity.

\section{The Rudin-Osher-Fatemi Model in MRW Spaces}\label{sect2}

\subsection{The $m$-ROF Model}

We consider the following   ROF-problem
 \begin{equation}\label{1RWROFmod}
\min \left\{ TV_m(u) + \frac{\lambda}{2} \int_X  \vert u(x) - f(x) \vert^2 d \nu(x)  : \ u \in L^2(X, \nu)  \right\},
\end{equation}
for $f\in L^2(X,\nu)$  and $\lambda>0$.
 We start by proving existence and uniqueness of the minimizer for problem \eqref{1RWROFmod} and a characterization.

  Let us write
  $$\mathcal{G}_m(u,f,\lambda) := TV_m(u) + \frac{\lambda}{2} \Vert u- f \Vert^2_{L^2(X, \nu)}.$$

 \begin{theorem}\label{ExistUniqmin} For any $f \in L^2(X, \nu)$ and $\lambda>0$, there exists a unique minimizer $u_\lambda$ of problem~\eqref{1RWROFmod}. Moreover, $u_\lambda$ is the unique solution of the  problem
 \begin{equation}\label{Ecuation1}
   \lambda (u - f) \in  \Delta_1^m (u).
 \end{equation}
Consequently, $u_\lambda \in L^2(X, \nu)$ is the solution of problem \eqref{1RWROFmod} if, and only if,  there exists $\hbox{\bf g}\in L^\infty(X\times X, \nu \otimes m_x)$ antisymmetric with $\Vert \hbox{\bf g} \Vert_\infty \leq 1$, such that
  \begin{equation}\label{trras01}\lambda (u_\lambda - f) =  {\rm div}_m \g
  \end{equation}
  and
        $$\g(x,y) \in {\rm sign}(u_\lambda(y) - u_\lambda(x)) \quad (\nu \otimes m_x)-a.e. \ (x,y) \in X \times X.$$
 \end{theorem}
  \begin{proof} Let $\{ u_n \}_{n\in N}$ be a minimizing sequence of problem \eqref{1RWROFmod}.
  Then,
  $$\alpha:= \inf \left\{ \mathcal{G}_m(u,f,\lambda) \ : \ u \in L^2(X, \nu) \right\} = \lim_{n \to \infty} \mathcal{G}_m(u_n,f,\lambda). $$
Since
  $$\Vert u_n \Vert^2_{L^2(X, \nu)} \leq 2 \left(\Vert u_n- f \Vert^2_{L^2(X, \nu)} + \Vert  f \Vert^2_{L^2(X, \nu)} \right)\leq 2 \left( \frac{2}{\lambda} \mathcal{G}_m(u_n,f,\lambda) + \Vert  f \Vert^2_{L^2(X, \nu)} \right),$$
  we have that $\{u_n\}_n$ is bounded in $L^2(X,\nu)$, and we can assume that, up to a subsequence,
  $$ u_n \to u_\lambda \quad \hbox{weakly in} \ L^2(X, \nu).$$
   Therefore, by the lower semi-continuity of the $L^2$-norm with respect to the   weak convergence and Lemma \ref{semicont}, we have that
   $$\mathcal{G}_m(u_\lambda,f,\lambda) \leq \liminf_{n \to \infty} \mathcal{G}_m(u_n,f,\lambda) = \alpha,$$
   and, consequently, $u_\lambda$ is a minimizer of problem \eqref{1RWROFmod}. Uniqueness follows from the strict convexity of $\Vert \cdot \Vert_{L^2(X, \nu)}^2$ and the convexity of $TV_m$.

   Since $u_\lambda$ is a minimizer of  problem \eqref{1RWROFmod}, we have that $0 \in \partial \mathcal{G}_m(u_\lambda,f,\lambda)$. Now, if $\Phi(u):= \frac{\lambda}{2} \Vert u- f \Vert^2_{L^2(X, \nu)}$, then, by  \cite[Corollary 2.11]{Brezis}, we have that
   $$\partial \mathcal{G}_m(u,f,\lambda) = \partial \mathcal{F}_m(u) + \partial \Phi(u),$$ thus
   $$0 \in \partial \mathcal{G}_m(u_\lambda,f,\lambda) = \partial \mathcal{F}_m(u_\lambda) + \lambda (u_\lambda - f),$$
   which gives~\eqref{Ecuation1}.
   Then, the characterization of $u_\lambda$  follows  from Theorem \ref{chsubd}.
\end{proof}

 Observe that, on account of~\eqref{perfect1}, we have that $ u_\lambda$ is the solution of problem \eqref{1RWROFmod} if, and only if,
\begin{equation}\label{ply002}
\left\{\begin{array}{l}
\displaystyle f-u_\lambda\in G_m(X),\\[8pt]
\Vert f-u_\lambda\Vert_{m^*}\le \frac{1}{\lambda} \ \hbox{ and}
\\[8pt] \displaystyle
   \lambda\int_X(f-u_\lambda)u_\lambda d\nu=TV_m(u_\lambda).
\end{array}\right.
\end{equation}

As a consequence of this we have:
\begin{proposition} Let $f \in L^2(X, \nu)$ and $\lambda >0$. Then,
 $u_\lambda =0$ is the solution of problem \eqref{1RWROFmod} if, and only if, $$f\in G_m(X)\quad\hbox{and}\quad
 \Vert f \Vert_{m,*} \leq \frac{1}{\lambda}.$$

 For $f\in G_m(X)$, if $\Vert f \Vert_{m,*} >\frac{1}{\lambda}$  then $u_\lambda$ is characterized by the following two conditions
\begin{equation}\label{porfin2}
\Vert f - u_\lambda\Vert_{m,*} = \frac{1}{\lambda}
\end{equation}
and
\begin{equation}\label{porfin2desdoblada}
 \lambda\int_X (f - u_\lambda) u_\lambda d \nu = TV_m(u_\lambda).
\end{equation}
\end{proposition}

\begin{proof}
  The first part follows from~\eqref{ply002}. Let   $f\in G_m(X)$ with $\Vert f \Vert_{m,*} >\frac{1}{\lambda}$.  From~\eqref{ply002} again,
      $$\Vert \lambda (f - u_\lambda) \Vert_{m,*} \leq 1 \quad \hbox{and} \quad \lambda\int_X (f -u_\lambda) u_\lambda d \nu = TV_m(u_\lambda).$$
      Now, since $\Vert f \Vert_{m,*} >\frac{1}{\lambda}$, we know that $u_\lambda \not\equiv 0$, thus, by \eqref{meyer1}, we have
      $$\Vert \lambda (f - u_\lambda) \Vert_{m,*}  \geq \frac{ \lambda}{TV_m(u_\lambda)} \int_X (f - u_\lambda) u_\lambda d \nu = 1.$$
      Therefore,
      $$\Vert   f - u_\lambda  \Vert_{m,*}  = \frac{1}{\lambda},$$
      and we conclude  the proof.
\end{proof}

The   $m$-ROF-model  leads to the $(BV,L^2)$-decomposition
\begin{equation}\label{BV2decompN}
f = u_\lambda + v_\lambda \ , \quad [u_\lambda , v_\lambda] = \displaystyle{\rm argmin} \left\{ TV_m(u) + \frac{\lambda}{2} \Vert v \Vert_2^2  \ : \ f=u+v \right\} .
\end{equation}
 Then, from the previous results, we can rewrite:
 \begin{corollary} Let $f \in L^2(X, \nu)$ and $\lambda >0$. For the $(BV,L^2)$-decomposition $[u_\lambda , v_\lambda]$  of~$f$, we have:

 \item[ (i)] $v_\lambda\in G_m(X)$, $\Vert v_\lambda\Vert_{m^*}\le\frac{1}{\lambda}$ \ and \ $\displaystyle\lambda\int_Xv_\lambda u_\lambda d\nu=TV_m(u_\lambda)$.

 \item[ (ii)]   $u_\lambda =0$ if, and only if,  $v_\lambda = f$.

 \item[ (iii)] For $f\in G_m(X)$,  if  $\Vert f \Vert_{m,*} > \frac{1}{\lambda}$ then
 $$\Vert v_\lambda \Vert_{m,*} = \frac{1}{\lambda}\quad \hbox{and} \quad  \lambda\int_X v_\lambda u_\lambda d \nu = TV_m( u_\lambda).$$
 \end{corollary}

\begin{remark}{\rm
(i) If $\lambda$ is too small then the regularization term $TV_m(u)$ is
excessively penalized and the image is over-smoothed, resulting in a loss of information in the
reconstructed image. On the other hand, if $\lambda$ is too large then the reconstructed image is
under-regularized and noise is left in the reconstruction.

\item[ (ii)]
 In~\cite{TNV1} and \cite{TNV2},  Tadmor,   Nezzar and Vese propose a multiscale decomposition in order to overcome the difficulties observed in the previous point. In this regard, the space of functions $G_m(X)$ is of particular interest, since, as we have seen in Corollary \ref{BV2decompN}, after a first decomposition $[u_\lambda, v_\lambda]$ the function $v_\lambda$ is a function of $G_m(X)$ which in turn can be decomposed. The multiscale decomposition takes advantage of this fact by taking an increasing sequence of scales $\lambda_i$ tending to infinity and inductively applying the $(BV,L^2)$-decomposition with scale parameter $\lambda_{i+1}$ to $v_{\lambda_i}$ so that after $k$-steps we have
$$f=\sum_{i=1}^k u_{\lambda_i}+v_{\lambda_k},\quad\Vert v_{\lambda_k}\Vert_{m^*}\le\frac{1}{\lambda_k}.$$
}\end{remark}

  Integrating both sides of \eqref{trras01} over $X$ and using Green's formula \eqref{Green}, with $u=1$ and $\z=\g$, on the right-hand side we get:
 \begin{proposition}\label{tau001}
If  $u_\lambda$ is the unique minimizer of problem \eqref{1RWROFmod} with  noisy image  $f$ then
$$\int_X u_\lambda(x) d\nu(x) = \int_X f(x) d\nu(x).$$
 \end{proposition}

 Furthermore, we have the following Maximum Principle.
 \begin{proposition}\label{MP1} If $[u_{i,\lambda} , v_{i,\lambda}]$  is  the $(BV,L^2)$-decomposition of $f_i$, $i=1,2$, then
 \begin{equation}\label{dom1001}
\Vert (u_{1,\lambda}- u_{2,\lambda})^+ \Vert_2 \leq  \Vert (f_1-f_2)^+  \Vert_2.
\end{equation}
 In particular, for  $c,\, C \in \R$, if $c \leq f \leq C$ $\nu$-a.e., and $[u_\lambda , v_\lambda]$ is  the $(BV,L^2)$-decomposition of $f$, then
 $$c \leq u_\lambda \leq C \quad \nu-a.e.$$
  \end{proposition}
 \begin{proof}
  Since $   \lambda ( u_{i,\lambda}-f_i) \in     \Delta_1^m(u_{i,\lambda})$, $i=1,2$,~\eqref{dom1001} is a direct consequence of the  complete accretivity of $ -\Delta_1^m$.

  The second part follows from~\eqref{dom1001}  and the fact  that, for a constant $k$,  $[k , 0]$ is  the $(BV,L^2)$-decomposition of $f=k$.
   \end{proof}

\begin{remark}\label{calib001}{\rm For the local ROF problem  in $\mathbb{R}^N$
 it is well known that if $f=\1_{B_r(0)}$ then the solution is given by
$$u_\lambda=\left\{
\begin{array}{ll}
0,&\hbox{for }0\le \lambda\le\frac{1}{2r},\\[8pt]
\left(1-\frac{1}{2r\lambda}\right)\1_{B_r(0)},&\hbox{for }\lambda >\frac{1}{2r}.
\end{array}
\right.
$$
For the  $m$-ROF model studied here, where the ambient space has finite measure, there does not exist a solution of the form $c\1_\Omega$ for $f=\1_\Omega$, whatever non-null measurable set $\Omega$, $\nu(\Omega)<\nu(X)$, is chosen. Indeed, if such a solution exists then, by
Proposition~\ref{tau001}, we would have $c=1$. Hence, by
Theorem \ref{ExistUniqmin}, we have that
$$0\in\Delta_1^m \1_\Omega,$$
which is impossible since we are assuming  ergodicity.
However, we can have a solution of the form $c\1_\Omega$ for particular functions $f=\1_\Omega+h\1_{X\setminus\Omega}$, where $\Omega$ is an $m$-calibrable set and $h$ satisfies
$$\int_{X\setminus\Omega}hd\nu=-\frac{1}{\lambda}P_m(\Omega).$$
Indeed, in this case we need to solve
    $$\lambda(c-1)\1_\Omega-\lambda h(x)\1_{X\setminus\Omega}\in\Delta_1^m \1_\Omega,$$
 and this implies (see~\cite[Remarks 5.6 \& 5.10]{MST1}) that: $c=1-\frac{\lambda_\Omega^m}{\lambda}$,   $\Omega$ is $m$-calibrable, and
$$\int_{X\setminus\Omega}hd\nu=-\frac{1}{\lambda}P_m(\Omega).$$
This is possible for $f(x)=\1_\Omega(x)-\frac{1}{\lambda}m_x(\Omega)\1_{X\setminus\Omega}(x)$.
}
\end{remark}

In the next result we construct a solution  of \eqref{1RWROFmod} for $f= b \overline{u}$, where $\overline{u}$ is a solution of $- u\in  \Delta_1^m u$.  Observe that, in this case, $\displaystyle \int_Xfd\nu=0$.

\begin{proposition} Let $\lambda , b > 0$. If $\overline{u}\in L^2(X,\nu)$ is a solution of
\begin{equation}\label{1Lpac1}
 - u \in  \Delta_1^m u,
\end{equation}
then
$u_\lambda=\left(b - \frac{1}{\lambda}\right)^+\overline{u}$ is a solution  of \eqref{1RWROFmod} with $f= b \overline{u}$. Conversely, if $\left(b - \frac{1}{\lambda}\right)\overline{u}$ is a solution of~\eqref{1RWROFmod} with $f= b \overline{u}$,   then $\overline{u}$ is a solution of \eqref{1Lpac1}.

\end{proposition}

\begin{proof} Set $a=\left(b - \frac{1}{\lambda}\right)^+$ and let $\overline{u}$ be a solution of \eqref{1Lpac1}. Suppose first that $b > \frac{1}{\lambda}$, so that $a= b - \frac{1}{\lambda}$. Then,
$$\lambda( a \overline{u} - b \overline{u}) = - \overline{u} \in  \Delta_1^m (\overline{u}) = \Delta_1^m (a \overline{u}).$$
Hence, by Theorem \ref{ExistUniqmin}, we have that $a\overline{u}$ is a solution of \eqref{1RWROFmod} with $f= b \overline{u}$. Now, assume that  $b \leq \frac{1}{\lambda}$, so that $a= 0$. Since $\overline{u}$ is a solution of \eqref{1Lpac1}, there exists $\hbox{\bf g}\in L^\infty(X\times X, \nu \otimes m_x)$ antisymmetric with $\Vert \hbox{\bf g} \Vert_\infty \leq 1$, such that
  $$ - {\rm div}_m \g = \overline{u}$$
  and
        $$\g(x,y) \in {\rm sign}(\overline{u}(y) - \overline{u}(x)) \quad \hbox{for } (\nu \otimes m_x)-a.e. \ (x,y) \in X \times X.$$
        If $\z:= \lambda b \g$, we have that $\Vert \z \Vert_\infty \leq 1$,
         $$ - \frac{1}{\lambda}{\rm div}_m \z = - b \,  {\rm div}_m \g =  b  \overline{u},$$
         and
        $$\z(x,y) \in {\rm sign}(0) \quad \hbox{for } (\nu \otimes m_x)-a.e. \ (x,y) \in X \times X.$$
        Therefore,
        $$ - \lambda b  \overline{u} \in \partial \mathcal{F}_m (0),$$
        and, by Theorem \ref{ExistUniqmin}, we have that $0$ is a solution of \eqref{1RWROFmod} with $f= b \overline{u}$.

        Suppose now that $a\overline{u}$ is a solution of \eqref{1RWROFmod} with $f= b \overline{u}$, and $b - a = \frac{1}{\lambda}$, then, by Theorem \ref{ExistUniqmin}, we have
         $$- \lambda (a\overline{u} - b \overline{u}) \in \partial \mathcal{F}_m (a\overline{u}).$$
         Hence, $\overline{u}$ is a solution of \eqref{1Lpac1}.
\end{proof}

 \begin{remark}\label{implicitR}{\rm In \cite{MST1} we have studied the total variational flow
\begin{equation}\label{TVF1}
u'(t) - \Delta_1^m u(t)\ni 0, \quad u(0) = u_0.
\end{equation}
There is a formal connection between the $m$-ROF-model \eqref{1RWROFmod} and the  total variational flow \eqref{TVF1} that can be drawn as follows. Given the initial datum $u_0$, we consider an implicit time discretization of the TVF \eqref{TVF1} using the following iterative procedure:
\begin{equation}\label{implicit}
   \frac{u_n - u_{n-1}}{\Delta t}  \in \Delta_1^m u_n  \quad n \in \N.
\end{equation}
Identifying the time step $\Delta t$ in \eqref{implicit} with the regularization parameter in \eqref{1RWROFmod}, that is, taking $\lambda = \frac{1}{\Delta t}$, we observe that
each iteration in \eqref{implicit} can be equivalently
approached by solving \eqref{1RWROFmod} (see~\eqref{Ecuation1}), where we take $u = u_n$ and $f= u_{n-1}$.  In  the next Section we discuss how to solve the $m$-ROF model via the gradient descent method.
}
\end{remark}

\subsection{The Gradient Descent Method}\label{sectgdm}\
   As in \cite{ROF},  we can see that problem \eqref{1RWROFmod} is well-posed  by using the gradient descent method.  For this, one needs to solve the Cauchy problem
\begin{equation}\label{ParabROFmodN}
 \left\{ \begin{array}{ll} v_t \in \Delta_1^m v(t)- \lambda (v(t) - f) \quad &\hbox{in} \ (0, T) \times X \\ \\   v(0, x) = v_0(x) \quad &\hbox{in} \ x \in X,\end{array} \right.
 \end{equation}
 with $v_0$ satisfying
 $$\int_\Omega v_0 = \int_\Omega f.
 $$
Now, problem \eqref{ParabROFmodN}  can be rewritten as the following abstract Cauchy problem in $L^2(X,\nu)$:
\begin{equation}\label{ACP2}
v'(t) + \partial \mathcal{G}_m(v(t),f,\lambda) \ni 0, \quad v(0) = v_0.
\end{equation}
 Then, since $\mathcal{G}_m( \cdot,f,\lambda)$ is convex and lower semi-continuous,  by the theory of maximal monotone operators (\cite{Brezis}), we have that, for any initial data $v_0 \in L^2(X,\nu)$, problem \eqref{ACP2} has a unique strong solution. Therefore, if we define a solution of problem  \eqref{ParabROFmodN} as a function $v \in C(0, T; L^2(X,\nu)) \cap W^{1,1}_{loc}(0, T; L^2(X,\nu))$ such that $v(0, x) = v_0(x)$ for $\nu$-a.e. $x \in X$ and satisfying
$$ \lambda(v(t) - f) + v_t(t) \in \Delta_1^m (v(t)) \quad \hbox{for almost every} \ \ t \in (0, T),$$
we have the following existence and uniqueness result.

\begin{theorem}\label{ExistUniqROF}  For every $v_0 \in L^2( X,\nu)$ there exists a unique strong  solution of the Cauchy problem~\eqref{ParabROFmodN} in $(0,T)$ for any $T > 0$.
Moreover, we have the following contraction and maximum principles in any $L^q(X,\nu)$--space, $1\le q\le \infty$:
\begin{equation}\label{contrprin}\Vert (v(t)-w(t))^+\Vert_q\le \Vert (v_0-w_0)^+\Vert_q\quad \forall\, 0<t<T,\end{equation}
for any pair of solutions $v,\, w$ of problem~\eqref{ParabROFmodN} with initial data $v_0$, $w_0$ and noisy images $f$ and $\hat f$, with $f\le \hat f$, respectively.
\end{theorem}

Note that the contraction principle \eqref{contrprin} in any $L^q$-space follows from the fact that the operator $\partial \mathcal{G}_m (\cdot,f,\lambda)$ is completely accretive.

\begin{theorem}\label{sab1001}  For  $f \in L^2(X, \nu)$, let $(T_\lambda(t))_{t \geq 0}$ be the semigroup solution of the Cauchy problem~\eqref{ParabROFmodN}. Then, for every $v_0 \in L^2(X, \nu)$, we have
\begin{equation}\label{asymp1}
\Vert T_\lambda(t)v_0 - u_\lambda \Vert_2 \leq \Vert v_0 -   u_\lambda  \Vert_2 \ e^{-\lambda t} \quad \hbox{for all} \ \ t \geq 0,
\end{equation}
where $   u_\lambda $ is the unique minimizer of problem \eqref{1RWROFmod} with this same $f$.
\end{theorem}
\begin{proof}  If $v(t):= T_\lambda(t) v_0$, we have
 \begin{equation}\label{0Ecuation1}   v_t + \lambda (v(t) - f) \in     \Delta_1^m(v(t)),  \end{equation}
  and, by Theorem \ref{ExistUniqmin},
\begin{equation}\label{1Ecuation1}
   \lambda (  u_\lambda - f) \in    \Delta_1^m (   u_\lambda ).
 \end{equation}
Now, since $ -\Delta_1^m$ is a monotone operator in $L^2(X, \nu)$, we get
$$\int_X (v(t) -    u_\lambda )(- v_t - \lambda (v(t) - f) - (  - \lambda (  u_\lambda  - f)) d\nu \geq 0,$$
from where it follows that
$$\frac12 \frac{d}{dt}\int_X (v(t) -  u_\lambda )^2 d\nu + \lambda \int_X (v(t) -  u_\lambda )^2d\nu \leq 0.$$
Then, integrating this ordinary differential inequality, we obtain \eqref{asymp1}.
\end{proof}

\begin{proposition}\label{consermmaas}   Let $(T_\lambda(t))_{t \geq 0}$ be the semigroup solution of the Cauchy problem~\eqref{ParabROFmodN} and let  $v_0 \in L^2(X, \nu)$ satisfying  $\displaystyle \int_X v_0 d\nu =  \int_X f   d\nu$. Then,
$$\int_X T_\lambda(t) v_0 d\nu = \int_X f d\nu \quad \hbox{for all} \ t \geq 0.$$
\end{proposition}

\begin{proof} If $v(t):= T_\lambda(t) v_0$, we have
$$0 \in \lambda (v(t) - f) + v_t - \Delta_1^m  v(t) . $$
Integrating and having in mind that $\displaystyle\int_X v_0 d\nu =  \int_X f \ d\nu$,  we get
$$\lambda \int_0^t \int_X v(s)  d\nu ds - \lambda t \int_X f d\nu + \int_X v(t) d\nu - \int_X f d\nu =0.$$
Then, the function
$$z(t):= \int_0^t \int_X v(s)  d\nu ds,$$
verifies
$$\left\{\begin{array}{ll} \displaystyle z'(t) + \lambda z(t) = \lambda t \int_X f d\nu + \int_X f d\nu \\[7pt] z(0) = 0, \end{array} \right.$$
whose unique solution is $\displaystyle z(t) = t \int_X f d \nu$. Hence
$$\int_X v(t) d\nu = \int_X f d\nu.$$
\end{proof}

\section{The $m$-ROF-Model with $L^1$-fidelity term}

We will study in this section the ROF-model with $L^1$-fidelity term, that is, given $f \in L^1(X, \nu)$  and $\lambda>0$, we will study the  minimization of the energy given by the sum of the total variation and the $L^1$-fidelity term:
\begin{equation}\label{l1fidelity1}
\mathcal{E}_m(u, f, \lambda):= TV_m(u) + \lambda \int_X \vert u - f \vert d\nu, \ \  u\in L^1(X,\nu).
\end{equation}

Let us denote
$$\mathcal{E}_m( f, \lambda):=\inf_{u \in L^1(X,\nu)} \mathcal{E}_m(u, f, \lambda).$$
 We introduce the following notation to denote the set of minimizers of $\mathcal{E}_m(\cdot, f, \lambda)$ for a given function $f \in L^1(X, \nu)$ and $\lambda >0$:
  $$M(f, \lambda) = \left\{ u \in L^1(X, \nu) \, : \, \mathcal{E}_m(u, f, \lambda) = \mathcal{E}_m( f, \lambda)  \right\}.$$

Note that the set $M(f, \lambda)$ can have several elements. Due to the convexity and the lower semi-continuity of the energy  functional $\mathcal{E}_m(\cdot, f, \lambda)$ we have that

\begin{equation}\label{Conclo}
\hbox{the set} \ M(f, \lambda) \ \hbox{is closed and convex in} \ L^1(X,\nu).
\end{equation}

In the local case, that is, for problem \eqref{BV1decomp}, the fact that there exists a minimizer for every data in $L^1$  is a consequence of the direct method of the calculus of variations. However, in our context, we do not have sufficient compactness properties in order to apply this method. Therefore, the proof that $M(f, \lambda) \not= \emptyset$ for every $f \in L^1(X, \nu)$ will be shown after the study of the geometric problem associated to the $(BV,L^1)$-decomposition is adressed in Section~\ref{secgeoprobl}.

  We have the following Maximum Principle.
 \begin{proposition}\label{MP2}
  Let $f\in L^1(X,\nu)$, $\lambda>0$ and $c, C \in \R$, and assume that $c \leq f \leq C$ $\nu$-a.e.  Then
 $$c \leq u  \leq C \quad \nu-a.e. \quad \forall u \in M(f, \lambda).$$
 \end{proposition}

 \begin{proof} Let $u \in M(f, \lambda)$.
 Obviously, we have that $TV_m( u  \wedge C ) \leq TV_m(u )$ and $\vert  ( u  \wedge C ) - f \vert \leq \vert u  - f \vert$. Hence,
 $$TV_m( u  \wedge C ) +  {\lambda} \Vert  (u  \wedge C)  - f \Vert_1 \leq TV_m(u ) +  {\lambda}  \Vert u  - f \Vert_1.$$
 Therefore, the above inequality is an equality and we have $\Vert  (u  \wedge C)  - f \Vert_1 =  \Vert u  - f \Vert_1$, from where we conclude that $$u \wedge C  = u  \quad \nu\hbox{-a.e.}$$ Similarly, it follows that $ u  \vee c = u \ $ $\nu$-a.e.
\end{proof}

\begin{remark}\label{saltos1}{\rm
Similarly to~\cite[Claims 4 \&  5]{ChanEsedoglu}, if $\lambda_2>\lambda_1>0$, then
$$\Vert u_{\lambda_1}-f\Vert_1\ge \Vert u_{\lambda_2}-f\Vert_1,\quad\hbox{for }u_{\lambda_i} \in M(f, \lambda_i),\ i=1,2,$$
and the set
$$\Lambda(f):=\left\{\lambda:\inf_{u\in M(f,\lambda)}\Vert u-f\Vert_1<\sup_{u\in M(f,\lambda)}\Vert u-f\Vert_1\right\}$$
is at most countable.
}
\end{remark}

\begin{proposition}\label{sec002}
 Let $f\in L^1(X,\nu)$, then $\mathcal{E}_m( f, \lambda)$ is Lipschitz continuous with respect to $\lambda$.
\end{proposition}
\begin{proof}
 Since $\mathcal{E}_m( f, \lambda)$ is defined as the pointwise infimum of a collection of increasing and linear functions in $\lambda$, we have that $\mathcal{E}_m( f, \lambda)$ is increasing and concave in $\lambda$. This, together with the fact that
$$\mathcal{E}_m( f, \lambda)\leq \mathcal{E}_m(0, f, \lambda)=\lambda\Vert f\Vert_{1} \ ,$$
gives the desired property.
\end{proof}

\begin{lemma}\label{sec001}
Let $\Omega$ be a $\nu$-measurable set  and $\lambda>0$.  Then
$$u_\lambda\in M(\1_\Omega,\lambda) \ \Longleftrightarrow \   \1_X-u_\lambda\in M(\1_{X\setminus\Omega},\lambda). $$
\end{lemma}
\begin{proof} This follows easily since
$\mathcal{E}_m(u,\1_\Omega,\lambda)=\mathcal{E}_m(  \1_X-u,\1_{X\setminus\Omega},\lambda) . $
\end{proof}

In the next result we obtain the Euler-Lagrange equation of the variational problem~\eqref{VARPRO1}.

 \begin{theorem}\label{El1} Assume that $f \in L^2(X, \nu)$ and let $\lambda>0$. Then, $u_\lambda \in M(f, \lambda)$ if, and only if, \ there exists $\xi \in  \hbox{sign}(u_\lambda - f)$ such that   $$\lambda \xi \in \Delta_1^m (u_\lambda).$$
 \end{theorem}
  \begin{proof}
  We have that $\mathcal{E}_m(u, f, \lambda) = \mathcal{F}_m(u) + \lambda\mathcal{G}_f(u)$, with
 $$\mathcal{G}_f(u):=  \int_X \vert u - f \vert d\nu.$$
Moreover,  $u_\lambda \in M(f, \lambda)$ if, and only if, $0 \in \partial \mathcal{E}_m(u_\lambda, f, \lambda)$.
Now, by  \cite[Corollary 2.11]{Brezis}, we have that
   $$\partial \mathcal{E}_m(u) = \partial \mathcal{F}_m(u) +  \lambda \partial \mathcal{G}_f(u),$$
   and then
$$u_\lambda \in M(f, \lambda)  \iff 0 \in \partial \mathcal{F}_m(u_\lambda) +   \lambda \partial \mathcal{G}_f(u_\lambda).$$
Now, it is not difficult to see that  $$v \in \partial \mathcal{G}_f(u_\lambda) \iff v \in  \hbox{sign}(u_\lambda - f).$$
Therefore,
$$u_\lambda \in M(f, \lambda)  \iff \hbox{there exists}  \ \xi \in  \hbox{sign}(u_\lambda - f) \ \hbox{ such that} \   \lambda \xi \in \Delta^m_1 (u_\lambda).$$
\end{proof}

\begin{remark}\label{aboverem01}{\rm
     As a consequence of Theorem~\ref{El1} and Theorem \ref{chsubd}, we have that:
\\[6pt] $u_\lambda \in M(f, \lambda)$ if, and only if, there exists  $\xi \in  \hbox{sign}(u_\lambda - f)$ and $\hbox{\bf g}\in L^\infty(X\times X, \nu \otimes m_x)$   antisymmetric with $\Vert \hbox{\bf g} \Vert_\infty \leq 1$  satisfying
$$\int_{X}\g(x,y)\,dm_x(y)= \lambda  \xi(x) \quad \hbox{ for } \nu-\mbox{a.e } x\in X, \ \hbox{ and}$$
$$\hbox{\bf g}(x,y) \in {\rm sign}(u_\lambda(y) - u_\lambda(x)) \quad \hbox{ for }  (\nu \otimes m_x)-a.e. \ (x,y) \in X \times X.$$
}\end{remark}

The $(BV, L^1)$-decomposion is contrast invariant (see~\cite{Darbon} for the   continuous case):
\begin{corollary}\label{continv}  Let $f\in L^1(X,\nu)$, $\lambda>0$ and $T : \R \rightarrow \R$ be a nondecreasing function. Then, if $u_\lambda \in M(f,\lambda)$, we have that $T(u_\lambda) \in M(T(f),\lambda)$.
\end{corollary}
\begin{proof} Given $u_\lambda \in M(f,\lambda)$ we have that, by Theorem \ref{El1}, there exists $\xi \in  \hbox{sign}(u_\lambda - f)$ such that   $$\lambda \xi \in \Delta_1^m (u_\lambda).$$
Then, since $T$ is nondecreasing, we have that $\xi \in  \hbox{sign}(T(u_\lambda) - T(f))$ and   $$\lambda \xi \in \Delta_1^m (T(u_\lambda)).$$
Therefore, applying again Theorem \ref{El1}, we get $T(u_\lambda) \in M(T(f),\lambda)$.
\end{proof}

Like in the local case, by the coarea   and the  layer cake  formulas, we obtain that
$$\int_X \vert u - f \vert d\nu = \int_{-\infty}^{+\infty}  \nu( \{ x :  u(x) > t \} \bigtriangleup \{ x :  f(x) > t \}) dt$$
where $$A \bigtriangleup B:= (A \setminus B) \cup (B \setminus A).$$
Therefore, the energy functional $\mathcal{E}_m(\cdot, f, \lambda)$ can be rewritten in a geometric form in terms of the energies of the superlevel sets of $u$ as follows (see \cite[Proposition 5.1]{ChanEsedoglu}).
\begin{theorem}\label{setformulation}  Let $u$, $f\in L^1(X,\nu)$ and $\lambda>0$, then
\begin{equation}\label{esetformulation}
\mathcal{E}_m(u, f, \lambda) = \int_{-\infty}^{+\infty}  \Big(P_m (E_t(u)) + \lambda \nu( E_t(u) \bigtriangleup E_t(f))\Big) dt.
\end{equation}
\end{theorem}

Consequently, given a $\nu$-measurable set $\Omega\subset X$ and taking $f= \1_\Omega$, by the Maximum Principle (Proposition \ref{MP2}), we get
\begin{equation}\label{esetformulationN}
\mathcal{E}_m(u, \1_\Omega, \lambda) = \int_{0}^{1}  \Big(P_m (E_t(u)) + \lambda \nu( E_t(u) \bigtriangleup \Omega)\Big) dt.
\end{equation}

\subsection{The Geometric Problem}\label{secgeoprobl}

Given a $\nu$-measurable set $F \subset X$, we consider the geometric functional
$$\mathcal{E}_m^G(A,F,\lambda):= P_m(A) + \lambda\nu(A \bigtriangleup F).$$
In view of   Theorem~\ref{setformulation}, given $f \in L^1(X, \nu)$, one may consider the family of geometric problems
\begin{equation}\label{geomprob1}
P(f, t, \lambda): \quad \inf_{ A \in \mathcal{B}(X)} \mathcal{E}_m^G(A,E_t(f),   \lambda), \  \hbox{ for any } t \in \R,
\end{equation}
where $\mathcal{B}(X)$ is the family of Borel subsets of $X$.
By Theorem~\ref{setformulation},  we can see that
 a minimizer of
$\mathcal{E}_m^G(A,F,\lambda)$, $A \in \mathcal{B}(X)$,
  always exists:

 \begin{theorem}\label{characteristicf}   Let    $F\subset X$ be a non- $\nu$-null $\nu$-measurable set and $\lambda>0$. Then, there exists a minimizer $u_\lambda$ of $\mathcal{E}_m(., \1_F, \lambda)$. Moreover, for almost every $t \in ]0,1[$,   $ E_t(u_\lambda) $ is   a minimizer of $\mathcal{E}_m^G(., F, \lambda)$, and
 $$\min_{ A \in \mathcal{B}(X)} \mathcal{E}_m^G(A,F, \lambda)=\min_{u \in L^1(X,\nu)} \mathcal{E}_m(u, \1_F, \lambda).$$
 \end{theorem}
\begin{proof} Since $\1_F\in L^\infty(X,\nu)$,  by the direct method of the calculus of variations,  we have that there exists $u_\lambda$   (which, by Proposition~\ref{MP2}, belongs to $L^\infty(X,\nu)$)  such that
$$\mathcal{E}_m(u_\lambda,\1_F, \lambda) =\min_{u \in L^1(X,\nu)} \mathcal{E}_m(u, \1_F, \lambda).$$
Now, by Theorem~\ref{setformulation},
$$\int_0^1 \mathcal{E}_m^G(E_t(u_\lambda),F, \lambda)dt=\mathcal{E}_m(u_\lambda,\1_F, \lambda)
\le\inf_{ A \in \mathcal{B}(X)} \mathcal{E}_m(\1_A,\1_F, \lambda)=\inf_{ A \in \mathcal{B}(X)} \mathcal{E}_m^G(A,F, \lambda),
$$
hence, for a.e. $t \in]0,1[$,
$$  \mathcal{E}_m^G(E_t(u_\lambda),F, \lambda) =\inf_{ A \in \mathcal{B}(X)} \mathcal{E}_m^G(A,F, \lambda),
$$
which concludes the proof.
\end{proof}

  The next results  were proved in \cite{YGO} for the local case. Now, since  their proofs only use properties of a measure,  the submodularity of the perimeter, that we have for the $m$-perimeter (Proposition~\ref{Juli1sub}), and the local version of Theorem \ref{setformulation}, with the same proofs  one can obtain these results:

\begin{lemma}\label{ole1} Given $f \in L^1(X, \nu)$  and $\lambda>0$ there exists a function $u \in L^1(X, \nu)$ such that
$$  \mathcal{E}_m^G(E_t(u),E_t(f), \lambda) =\inf_{ A \in \mathcal{B}(X)} \mathcal{E}_m^G(A,E_t(f), \lambda)  \quad \hbox{for all} \ t \in \R.
$$
\end{lemma}

\begin{theorem}\label{thin001}
 For $f\in L^1(X,\nu)$  and $\lambda>0$ there exists (at least) one minimizer of the variational problem
\begin{equation}\label{VARPRO1}
\min_{u \in L^1(X,\nu)} \mathcal{E}_m(u, f, \lambda).
\end{equation}
\end{theorem}
\begin{proof} Let $u$ be the function obtained in Lemma \ref{ole1}. Then,  by Theorem \ref{setformulation}, given $v\in L^1(X,\nu)$, we have
$$\mathcal{E}_m(u, f, \lambda) = \int_{-\infty}^{+\infty} \mathcal{E}_m^G(E_t(u),E_t(f), \lambda) dt \leq \int_{-\infty}^{+\infty} \mathcal{E}_m^G(E_t(v),E_t(f), \lambda) dt = \mathcal{E}_m(v, f, \lambda).$$
\end{proof}

Duval, Aujol and Gousseau in \cite[Theorem 4.2]{DAG} show that, for the continuous case, there is an equivalence with the geometric problem. This result  can be extended,  on account of the submodularity proved in Proposition~\ref{Juli1sub} for the $m$-perimeter,
to our nonlocal context:

\begin{theorem}\label{geomequiv} Let $f \in L^1(X, \nu)$  and $\lambda>0$.  The following assertions are equivalent:
\item[  (i)] $u$ is a solution of Problem \eqref{VARPRO1}.

\item[ (ii)] Almost every level set $E_t(u)$ is a solution of \eqref{geomprob1}.

\end{theorem}

  In \cite[Proposition~5.5]{DAG} it is also shown that at points where the boundary of a minimizer of the geometric problem for datum $F$ and fidelity parameter $\lambda$ does not coincide with the boundary of $F$, the mean curvature  is $\pm\lambda$. Let us see that there is a nonlocal counterpart of this fact   but where the nonlocal character of the problem  gives rise to a nontrivial extension.

We recall the concept of $m$-mean curvature introduced in~\cite{MST1}.
\begin{definition}{\rm
Let $E \subset X$ be $\nu$-measurable. For a point $x  \in X$ we define  the {\it $m$-mean curvature of $\partial E$ at $x$} as
\begin{equation}\label{defcurdefdef}H^m_{\partial E}(x):= \int_{X}  (\1_{X \setminus E}(y) - \1_E(y)) dm_x(y).\end{equation}}
\end{definition}
Observe that
\begin{equation}\label{defcurp01}H^m_{\partial E}(x) =  -H^m_{\partial (X\setminus E)}(x),\end{equation}
and
\begin{equation}\label{defcur}H^m_{\partial E}(x) =  1 - 2 \int_E  dm_x(y).\end{equation}

 \begin{proposition}\label{mmeancurvatureandlambda}
Let $F\subset X$ be a $\nu$-measurable set with $ 0 < \nu(F) <\nu(X)$,  $\lambda>0$, and let $E$ be a minimizer of $\mathcal{E}_m^G(\cdot, F, \lambda)$. Let $A\subset X$ be a non-null $\nu$-measurable set, we have:

\item{ (1)} For  $\nu(A\cap F)=0$,

\item{\qquad (i) } if $\nu(A\setminus E)>0$,
$$\frac{1}{\nu(A\setminus E)}\int_{A\setminus E}H_{\partial E}^m(x)d\nu(x) \geq -  \lambda \ + \ \frac{1}{\nu(A\setminus E)}\int_{A\setminus E}m_x(A\setminus E)d\nu(x)  \ge -\lambda$$

\item{\qquad (ii) } if $\nu(A\cap E)>0$,
$$\frac{1}{\nu(A\cap E)}\int_{A\cap E}H_{\partial E}^m(x)d\nu(x) \le  - \lambda \ - \ \frac{1}{\nu(A\cap E)}\int_{A\cap E}m_x(A\cap E)d\nu(x)  \le -\lambda.$$

\item{ (2)} For   $\nu(A\setminus F)=0$,

\item{\qquad (i) } if $\nu(A\setminus E)>0$,
$$\frac{1}{\nu(A\setminus E)}\int_{A\setminus E}H_{\partial E}^m(x)d\nu(x) \geq  \, \lambda \ + \ \frac{1}{\nu(A\setminus E)}\int_{A\setminus E}m_x(A\setminus E)d\nu(x)\ge \lambda .$$

\item{\qquad (ii) } if $\nu(A\cap E)>0$,
$$\frac{1}{\nu(A\cap E )}\int_{A\cap E }H_{\partial E}^m(x)d\nu(x)  \le \, \lambda \   - \ \frac{1}{\nu(A\cap E )}\int_{A\cap E }m_x(A\cap E )d\nu(x) \le \lambda.$$

\end{proposition}

\begin{proof}  {\it  (i)}:
 Since $E$ is a minimizer of $\mathcal{E}_m^G(\cdot, F, \lambda)$ we have that
$$P_m(E)+\lambda\nu(E\triangle F)\leq P_m(E\cup A)+\lambda\nu((E\cup A)\triangle F).$$
In other words,
$$\int_E\int_{X\setminus E}dm_x(y)d\nu(x)-\int_{E\cup A}\int_{X\setminus (E\cup A)}dm_x(y)d\nu(x)\le \lambda\left(\nu((E\cup A)\triangle F)-\nu(E\triangle F)\right),$$
but
$$\nu((E\cup A)\triangle F)-\nu(E\triangle F)=\left\{\begin{array}{cc}
    \nu(A\setminus E) & \hbox{if } \nu(A\cap F)=0 , \\[10pt]
    -\nu(A\setminus E) & \hbox{if } \nu(A\setminus F)=0 ;
  \end{array}\right.
$$
and
$$\begin{array}{l}\displaystyle\int_E\int_{X\setminus E}dm_x(y)d\nu(x)-\int_{E\cup A}\int_{X\setminus (E\cup A)}dm_x(y)d\nu(x)
\\ \\ \displaystyle  \qquad\qquad =\int_E\int_{A\setminus E}dm_x(y)d\nu(x)-\int_{A\setminus E}\int_{X\setminus E}dm_x(y)d\nu(x)+\int_{A\setminus E}\int_{A\setminus E}dm_x(y)d\nu(x)
\\ \\ \displaystyle \qquad\qquad  =\int_{A\setminus E}\int_Edm_x(y)d\nu(x)-\int_{A\setminus E}\int_{X\setminus E}dm_x(y)d\nu(x)+\int_{A\setminus E}\int_{A\setminus E}dm_x(y)d\nu(x)
\\ \\ \displaystyle
\qquad\qquad =-\int_{A\setminus E}H_{\partial E}^m(x)d\nu(x)+\int_{A\setminus E}m_x(A\setminus E)d\nu(x).
 \end{array}
 $$
Consequently:
\item{ \it (1)} If $\nu(A\cap F)=0$,  then $$\displaystyle \frac{1}{\nu(A\setminus E)}\int_{A\setminus E}H_{\partial E}^m(x)d\nu(x) \geq -  \lambda \ + \ \frac{1}{\nu(A\setminus E)}\int_{A\setminus E}m_x(A\setminus E)d\nu(x) .$$
\item{ \it (2)} If $\nu(A\setminus F)=0$, then
$$\displaystyle \frac{1}{\nu(A\setminus E)}\int_{A\setminus E}H_{\partial E}^m(x)d\nu(x) \geq  \, \lambda \ + \ \frac{1}{\nu(A\setminus E)}\int_{A\setminus E}m_x(A\setminus E)d\nu(x) .$$

 {\it (ii)}: These statements follow from {\it (i)} by~\eqref{defcurp01} and by taking into account that, since $P_m(E)=P_m(X\setminus E)$ and $E\triangle F=(X\setminus E)\triangle (X\setminus F)$, $E$ is a minimizer for $\mathcal{E}_m^G(\cdot, F, \lambda)$ if, and only if, $X\setminus E$ is a minimizer for $\mathcal{E}_m^G(\cdot, X\setminus F, \lambda)$, and, further,   that   $A\cap F=A\setminus (X\setminus F)$ and   $A\cap (X\setminus F)=A\setminus  F$.
\end{proof}

\begin{corollary}\label{mmeancurvatureandlambdaforgraphs}
Let $[X,d,m]$ be the metric random walk space associated to a connected and weighted discrete graph, and let $E$,   $F$ and $\lambda$ as in the hypothesis of Proposition \ref{mmeancurvatureandlambda}. Then
\item{ (1)}
\begin{equation}\label{gives01} \lambda \leq -H_{\partial E}^m(x) - \frac{w_{x,x}}{d_x}   \ \hbox{ for every } x\in E\setminus F,
\end{equation}
and
\begin{equation}\label{gives02}
 \lambda \leq H_{\partial E}^m(x)-   \frac{w_{x,x}}{d_x} \ \hbox{ for every } x\in F\setminus E.
\end{equation}
\item{ (2)}
\begin{equation}\label{gives03}\lambda \geq H_{\partial E}^m(x)+  \frac{w_{x,x}}{d_x}  \ \hbox{ for every } x\in E\cap F,
\end{equation}
and
\begin{equation}\label{gives04}
\lambda\geq -H_{\partial E}^m(x) + \frac{w_{x,x}}{d_x} \ \hbox{ for every } x\in X\setminus (E\cup F).
\end{equation}
 \end{corollary}

\begin{proof}
  {\it (1)}: If $E\setminus F\neq \emptyset$ let $x\in E\setminus F$ and take $A=\{x\}$, so that $A\cap E= A$. Note that $\nu(A)>0$ since $[X,d,m]$ is connected. Then, since $A\cap F=\emptyset$, by Proposition~\ref{mmeancurvatureandlambda} {\it (1)(ii)}, we get
    $$\frac{1}{\nu(\{x\})}\int_{\{x\}}H_{\partial E}^m(y)d\nu(y)  \le -\lambda -\frac{1}{\nu(\{x\})}\int_{\{x\}}m_y(\{x\})d\nu(y).$$
    That is,
    $  H_{\partial E}^m(x)\leq  -\lambda -m_x(\{x\}) \hbox{ for every } x\in E\setminus F,$ which gives~\eqref{gives01}.
    Now, \eqref{gives02}~can be obtained with a similar argument by using Proposition~\ref{mmeancurvatureandlambda} {\it (2)(i)}, or as follows: since $E$ is a minimizer for $\mathcal{E}_m^G(\cdot, F, \lambda)$ then $X\setminus E$ is a minimizer for $\mathcal{E}_m^G(., X\setminus F, \lambda)$. Consequently, from~\eqref{gives01},
     $  H_{\partial (X\setminus E)}^m(x)\leq  -\lambda - \frac{w_{x,x}}{d_x} \hbox{ for every } x\in (X\setminus E)\setminus (X\setminus F),$
     that is, since $ H_{\partial (X\setminus E)}^m(x)=-   H_{\partial E}^m(x)$,
     $H_{\partial E}^m(x) \geq \lambda +  \frac{w_{x,x}}{d_x} \ \hbox{ for every } x\in F\setminus E.$
     The proof of {\it (2)} is similar.
    \end{proof}

With this results at hand,
we obtain a priori estimates on the $\lambda$ for which a set $E$ can be a minimizer of $\mathcal{E}_m^G(\cdot, F, \lambda)$. Indeed,
 we must try with $\lambda$ such that
$$\begin{array}{c}\displaystyle
\max\left\{\sup_{x \in F\cap E}\left(H_{\partial E}^m(x)+  \frac{w_{x,x}}{d_x} \right),\sup_{x\notin F\cup E}\left(-H_{\partial E}^m(x)+  \frac{w_{x,x}}{d_x} \right)\right\}\le\lambda \qquad\qquad\qquad\qquad
\\ \\ \displaystyle
\qquad\qquad\qquad\qquad\le \min\left\{\inf_{x\in E\setminus F}\left(-H_{\partial E}^m(x)-  \frac{w_{x,x}}{d_x} \right),\inf_{x\in F\setminus E} \left(H_{\partial E}^m(x)-  \frac{w_{x,x}}{d_x} \right) \right\}.
\end{array}
$$

 \begin{definition}
Let $(X,d,\nu)$ be a metric measure space. For a $\nu$-measurable subset $E\subset X$ we will write $x\in\partial_\nu E$ if
\begin{equation} \nu( B_\epsilon(x)\cap E)>0 \ \hbox{ and } \  \nu( B_\epsilon(x)\setminus E)>0 \ \hbox{ for every } \epsilon>0.
\end{equation}
\end{definition}

\begin{corollary}\label{told001} Let $\Omega$ be a bounded domain in $\mathbb{R}^n$ and let $m=m^{J,\Omega}$ be the random walk  given in Example~\ref{dom001}(1). Suppose further that $\hbox{\rm supp} (J)=B_r(0)$. Let $E$, $F$ and $\lambda$ as in the hypothesis of Proposition \ref{mmeancurvatureandlambda} and suppose that $\partial_{\mathcal{L}^N} E$ is not empty. Let $x\in\partial_{\mathcal{L}^N} E $.
\item {  (i)}    If there is a neighborhood $V\subset \Omega$ of $x$ such that $\mathcal{L}^N(V\cap F)=0$ then $\mathcal{L}^N\left(B_r(x)\setminus\Omega\right)=0$
    and
$$ H_{\partial E}^m(x)=-\lambda.$$
\item {  (ii)}  If there is a neighborhood $V\subset \Omega$ of $x$ such that $\mathcal{L}^N(V\setminus F)=0$ then $\mathcal{L}^N\left(B_r(x)\setminus\Omega\right)=0$
    and
$$ H_{\partial E}^m(x) =\lambda.$$

Consequently, if $\mathcal{L}^N(B_r(x)\setminus\Omega)>0$ then $x\in\partial_{\mathcal{L}^N} F$.
\end{corollary}

  Observe that we can rewrite this result as follows:

{\it Let $\Omega$ be a bounded domain in $\mathbb{R}^n$ and $m=m^{J,\Omega}$ be the random walk given in Example~\ref{dom001}(1),
and let $E$, $F$ and $\lambda$ as in the hypothesis of Proposition \ref{mmeancurvatureandlambda}. For   $x\in\partial_{\mathcal{L}^N} E $, either $x\in \partial_{\mathcal{L}^N} F$ or, if $x\notin \partial_{\mathcal{L}^N} F$, then $ H_{\partial E}^m(x)=\pm\lambda$.}

\begin{proof}
{\it (i)}:  Let $x\in \partial_{\mathcal{L}^N} E$ such that we can find a neighborhood $V\subset \Omega$ of $x$ with $\mathcal{L}^N(V\cap F)=0$. Then, for $\epsilon>0$ small enough, we have that $B_\epsilon(x)\subset \Omega$ and    $\nu(B_\epsilon(x)\cap F)=0$. Hence, by Proposition~\ref{mmeancurvatureandlambda}~{\it (1)} with $A=B_\epsilon(x)$ and recalling the definition of $m^{J,\Omega}$,
$$\frac{1}{\mathcal{L}^N(B_\epsilon(x)\setminus E)}\int_{B_\epsilon(x)\setminus E}\left(\int_{\Omega\setminus E}J(z-y)dy-\int_EJ(z-y)dy\right)dz$$ $$ \geq    -\lambda  +   \frac{1}{\mathcal{L}^N(B_\epsilon(x)\setminus E)}\int_{B_\epsilon(x)\setminus E}\left(\int_{B_\epsilon(x)\setminus E}J(z-y)dy+\int_{\mathbb{R}^N\setminus\Omega}J(z-y)dy\right)dz$$
  $$ \geq    -\lambda  +   \frac{1}{\mathcal{L}^N(B_\epsilon(x)\setminus E)}\int_{B_\epsilon(x)\setminus E}\left(\int_{\mathbb{R}^N\setminus\Omega}J(z-y)dy\right)dz,$$
and
$$\frac{1}{\mathcal{L}^N(B_\epsilon(x)\cap E )}\int_{B_\epsilon(x)\cap E }\left(\int_{\Omega\setminus E}J(z-y)dy-\int_EJ(z-y)dy\right) dz $$
$$ \le  -\lambda \   - \ \frac{1}{\mathcal{L}^N(B_\epsilon(x)\cap E )}\int_{B_\epsilon(x)\cap E }\left(\int_{B_\epsilon(x)\cap E}J(z-y)dy+\int_{\mathbb{R}^N\setminus\Omega}J(z-y)dy\right)dz $$
$$ \le  -\lambda \   - \ \frac{1}{\mathcal{L}^N(B_\epsilon(x)\cap E )}\int_{B_\epsilon(x)\cap E }\left(\int_{\mathbb{R}^N\setminus\Omega}J(z-y)dy\right)dz .$$
Now, since, $\displaystyle z\mapsto \int_{\Omega\setminus E}J(z-y)dy-\int_EJ(z-y)dy$ and $\displaystyle z\mapsto  \int_{\mathbb{R}^N\setminus\Omega}J(z-y)dy$ are continuous, we get, by letting $\epsilon$ tend to $0$ in the above inequalities, that
$$\begin{array}{l}
\displaystyle
-\lambda-\int_{\mathbb{R}^N\setminus\Omega}J(x-y)dy \ge \int_{\Omega\setminus E}J(x-y)dy-\int_EJ(x-y)dy\\ \\
\displaystyle
\phantom{-\lambda-\int_{\mathbb{R}^N\setminus\Omega}J(x-y)dy}
\ge -\lambda+\int_{\mathbb{R}^N\setminus\Omega}J(x-y)dy.
\end{array}$$
Now, this implies that $\displaystyle\int_{\mathbb{R}^N\setminus\Omega}J(x-y)dy=0$, and
$$ H_{\partial E}^m(x)=\int_{\Omega\setminus E}J(x-y)dy-\int_EJ(x-y)dy=-\lambda.$$
A similar proof using Proposition~\ref{mmeancurvatureandlambda}~{\it (2)}  gives {\it (ii)}.
 \end{proof}

\begin{definition}
Let $[X,d,m]$ be a metric random walk space and $x\in X$. The random walk $m$ has the strong-Feller property at $x$ if
$$\lim_{y\to x}m_y(E)=m_x(E) \ \hbox{ for every $\nu$-measurable set $E$}.$$
\end{definition}

 Note that the examples of metric random walk spaces given in   Example \ref{JJ} (1--3) have the strong-Feller property. Consequently, some of the results of Corollary \ref{told001} will follow from the following results. However, since part of the Corollary is particular to the $m^{J,\Omega}$ random walk and, moreover, this random walk is of special importance in our development, we give these results separately.

\begin{lemma}\label{lemmadiflebFeller}
Let $[X,d,m]$ be a metric random walk space with invariant measure $\nu$, and let $x\in X$. Suppose that the random walk has the strong-Feller property at $x$, then, for a sequence of $\nu$-measurable sets $A_n\subset B(x,\frac1n)$ with $\nu(A_n)>0$, $n\in\N$, we have:
$$\lim_{n\to\infty}\frac{1}{\nu(A_n)}\int_{A_n}m_y(E)d\nu(y)=m_x(E) \hbox{ for any $\nu$-measurable set } E\subset X .$$

In particular, $$\lim_{n\to\infty}\frac{1}{\nu(A_n)}\int_{A_n}H_{\partial E}^m(y)d\nu(y)=H_{\partial E}^m(x) \hbox{ for any $\nu$-measurable set } E\subset X .$$
\end{lemma}
\begin{proof}
Let $E\subset X$ be a $\nu$-measurable set. Since the random walk has the strong-Feller property at $x$, there exists $n_0\in\N$ such that $|m_y(E)-m_x(E)|<\varepsilon$ for every $y\in B(x,\frac{1}{n_0})$. Then, for $n\ge n_0$, since $A_n\subset B(x,\frac1n)$, we have
$$\left| \frac{1}{\nu(A_n)}\int_{A_n}m_y(E)d\nu(y) - m_x(E) \right| \le \frac{1}{\nu(A_n)}\int_{A_n}|m_y(E)-m_x(E)|d\nu(y)<\varepsilon .$$

\end{proof}

\begin{proposition}
Let $[X,d,m]$ be a metric random walk space with  reversible  measure~$\nu$. Let $E$, $F$ and $\lambda$ as in the hypothesis of Proposition \ref{mmeancurvatureandlambda}. Let $x\in \partial_\nu E$ and suppose that the random walk has the strong-Feller property at $x$. The following holds:
\item{ (1)} If there is   a neighbourhood $V$ of $x$ such that $\nu(V\cap F)=0$, then
$$H_{\partial E}^m(x) =  - \lambda .$$
\item{ (2)} If there is   a neighbourhood $V$ of $x$ such that  $\nu(V\setminus F)=0$, then
 $$H_{\partial E}^m(x) =   \lambda .$$

 In particular, if $[X,d,m]$ has the strong-Feller property, then
 \item{ (1)}$$H_{\partial E}^m(x) =  - \lambda \ \hbox{ for every } x\in\partial_\nu E\cap \hbox{int}(X\setminus F).$$
\item{ (2)}  $$H_{\partial E}^m(x) =   \lambda \ \hbox{ for every } x\in\partial_\nu E\cap \hbox{int}(F).$$
\end{proposition}
\begin{proof}
  The proof follows by Proposition \ref{mmeancurvatureandlambda} and Lemma \ref{lemmadiflebFeller}. Indeed, let us prove \textit{(1)}. Take $A_n=B(x,\frac1n)\setminus E$, then, since $x\in \partial_\nu E$ we have that $\nu(A_n)>0$, and, moreover, since $V$ is a neighborhood of $x$, $\nu(A_n\cap F)\leq \nu(V\cap F)=0$ for $n$ large enough. Therefore, by \textit{(1)(i)} in Proposition \ref{mmeancurvatureandlambda}, we have
  $$\frac{1}{\nu(A_n)}\int_{A_n}H_{\partial E}^m(y)d\nu(y) \geq -\lambda \ \hbox{ for $n$ large enough},$$
  and taking limits when $n\to\infty$, by Lemma \ref{lemmadiflebFeller}, we get that $H_{\partial E}^m(x) \ge -\lambda $.

  To prove the opposite inequality we proceed equally by taking $A_n=B(x,\frac1n)\cap E$ and using \textit{(1)(ii)} in Proposition \ref{mmeancurvatureandlambda}.
\end{proof}

\subsection{Thresholding Parameters}
 In the local case it is well known (see~\cite{ChanEsedoglu}) that for $f= \1_{B_r(0)}$ the solution $u_\lambda$ of problem \eqref{BV1decomp} is given by:

\noindent {\it (i)} $u_\lambda = 0$ if $0< \lambda\le\frac{2}{r}$,

\noindent {\it (ii)} $ u_\lambda = c\1_{B_r(0)}$ with $0\le c\le 1$ if $\lambda=\frac{2}{r}$,

\noindent {\it (iii)} $u_\lambda = \1_{B_r(0)}$ if $\lambda\ge\frac{2}{r}$.

  \noindent In~\cite[Proposition 5.2]{DAG}  it is shown that this thresholding property holds true for a large class of calibrable sets in $\mathbb{R}^2$. Our goal now is to show   that there is also a thresholding property in the nonlocal case treated here.
\medskip

 For a constant $c$, we will abuse notation and denote the constant function $c\1_X$ by $c$ whenever this is not misleading.

  \begin{lemma}\label{proper001} Let $\lambda_0>0$.
\item{ (i)} If $f\in M(f,\lambda_0)$ then
$$\{f\}=  M(f,\lambda)\quad\forall \lambda>\lambda_0.$$
\item{ (ii)}  If $c  \in M(f,\lambda_0)$, $c\ge 0$ constant, then $c\in{\rm med}_\nu(f)$,
$${\rm med}_\nu(f) \subset   M(f,\lambda_0), $$
    and
$$    {\rm med}_\nu(f)=   M(f,\lambda)\quad\forall 0<\lambda<\lambda_0.$$
\item{ (iii)}   If $u\in M(f,\lambda_0)$ and $u\in M(f,\lambda_1)$ for some $u\in L^1(X,\nu)$ and $\lambda_0<\lambda_1$, then $u\in M(f,\lambda)$ for every $\lambda_0\le\lambda\le\lambda_1$.
    \end{lemma}

\begin{proof}
\item{ (i)}: Take $\lambda> \lambda_0$, then, for any $u\in L^1(X,\nu)$ such that $\nu(\{u\neq f\})>0$, we have
$$\mathcal{E}_m(f, f, \lambda)=TV_m(f)=\mathcal{E}_m(f, f, \lambda_0)\le
\mathcal{E}_m(u, f, \lambda_0)<\mathcal{E}_m(u, f, \lambda).$$

\item{ (ii)}:  Since $c \in M(f,\lambda_0)$ we have that, by Theorem \ref{El1}, there exists $\xi\in \hbox{sign}(c -f)$ and $\hbox{\bf g}\in L^\infty(X\times X, \nu \otimes m_x)$ antisymmetric with $\Vert \hbox{\bf g} \Vert_\infty \leq 1$ satisfying
$$\int_{X}\g(x,y)\,dm_x(y)= \lambda_0  \xi(x) \quad \hbox{for } \nu-\mbox{a.e }x\in X  \hbox{ and}$$
$$\hbox{\bf g}(x,y) \in {\rm sign}(0) \quad \hbox{for }(\nu \otimes m_x)-a.e. \ (x,y) \in X \times X.$$
Then,
$$ \int_X \xi d\nu(x)= \frac{1}{\lambda_0}\int_X \int_{X}\g(x,y)\,dm_x(y) d \nu(x) =0,$$
so that $0\in {\rm med}_\nu(c -f)$, which is equivalent to $c\in{\rm med}_\nu(f)$.
Now, for $\lambda<\lambda_0$, taking $g_\lambda(x,y)=\frac{\lambda}{\lambda_0}g(x,y)$ we obtain that
\begin{equation}\label{ply001}c\in M(f,\lambda).
\end{equation}

  Furthermore, by \eqref{NNsigg1}, for any other $m\in \hbox{med}_\nu(f)$ and any $\lambda>0$,
$$ \mathcal{E}(c,f,\lambda_0)=\lambda\int_X|c-f|d\nu=\lambda\int_X|m-f|d\nu=  \mathcal{E}(m,f,\lambda),$$
so that
$$\hbox{med}_\nu(f)\subset M(f,\lambda_0), \quad\forall 0<\lambda<\lambda_0.$$

 Now, let $m\in \hbox{med}_\nu(f)$, for any constant function $k\notin \hbox{med}_\nu(f)$, by \eqref{NNsigg1} we have that
$$\int_X|k-f|d\nu>\int_X|m-f|d\nu$$
so $k\notin M(f,\lambda)$ for every $\lambda>0$.

Suppose then that there exists some nonconstant function $u$, such that $u\in M(f,\lambda)$ for $0<\lambda<\lambda_0$. Since $\nu$ is ergodic with respect to $m$ we have that $TV_m(u)>0$, thus
$$\mathcal{E}(u,f,\lambda)\leq\mathcal{E}(m,f,\lambda)$$
implies that
$$\int_X|u-f|d\nu<\int_X|m-f|$$
and, therefore,
$$\mathcal{E}(u,f,\lambda_0)=\mathcal{E}(u,f,\lambda)+(\lambda_0-\lambda)\int_X|u-f|d\nu < \qquad \qquad \qquad \qquad$$
$$\qquad \qquad \qquad \qquad <\mathcal{E}(m,f,\lambda)+(\lambda_0-\lambda)\int_X|m-f|d\nu=\mathcal{E}(m,f,\lambda_0)$$
which is a contradiction.
Consequently,
$$\hbox{med}_\nu(f)= M(f,\lambda)\quad\forall 0<\lambda<\lambda_0.$$

\item{ (iii)}  Follows easily.
\end{proof}

 \begin{proposition}\label{proper002}  Let $(\lambda_0, u_0)$  be an  $m$-eigenpair of the $1$-Laplacian $\Delta^m_1$ on $X$ with $\lambda_0>0$. Then, $0\in  \hbox{med}_\nu(u_0)$ and
 $$\left\{ \begin{array}{lll}  \hbox{med}_\nu(u_0)=M(u_0,\lambda) \quad &\hbox{ if} \ \ 0 < \lambda < \lambda_0, \\ \\ \{ c u_0 \, : \, 0 \leq c \leq 1 \}\cup\hbox{med}_\nu(u_0)  \subset M(u_0, \lambda_0)
 \quad & \
 \\ \\ \{ u_0 \}  = M(u_0, \lambda) \quad &\hbox{ if} \ \  \lambda > \lambda_0.
 \end{array} \right.$$
\end{proposition}

\begin{proof}
Since $(\lambda_0, u_0)$ is an $m$-eigenpair of the $1$-Laplacian $\Delta^m_1$ with $\lambda_0>0$, we have that $0\in  \hbox{med}_\nu(u_0)$ (see \cite[Corollary 6.11]{MST1}). Furthermore, by the definition of $m$-eigenpair, we have that
$$\hbox{there exists} \ \xi_0 \in \hbox{sign}(u_0) \ \hbox{ such that }  - \lambda_0 \xi_0 \in \Delta_1^m(u_0).$$
Hence, for $0 < c \leq 1$, $\xi := - \xi_0 \in \hbox{sign}(c u_0 -u_0)$ and  $\lambda_0 \xi \in \Delta_1^m(u_0) = \Delta^m_1(cu_0)$, which implies that $cu_0 \in M(u_0, \lambda_0)$. Moreover, since $TV_m(u_0)=\lambda_0$ (see \cite[Remark 6.2]{MST1}) and $\Vert u_0\Vert_1 =1$, we have that
$$\mathcal{E}(u_0,u_0,\lambda_0)=\lambda_0=\mathcal{E}(0,u_0,\lambda_0).$$
Consequently,  by Lemma~\ref{proper001},  we get the rest of the thesis.

\end{proof}

 \medskip

In \cite[Proposition 6.12] {MST1} we showed that  if  $(\lambda_0, u_0)$  is an  $m$-eigenpair of the $1$~-~Laplacian $\Delta^m_1$ then $\left(\lambda_0, \frac{1}{\nu(E_0(u))}\1_{E_0(u)}\right)$  is also an  $m$-eigenpair of $\Delta^m_1$, where $E_0(u)$ is the upper $0$-level set of $u$. Moreover $E_0(u)$ is an $m$-calibrable set and $\lambda_0=\lambda_{E_0(u)}^m$. Then, as a consequence of the previous result we obtain the following.

\begin{corollary}\label{proper003}  Let $\Omega\subset X$ be a $\nu$-measurable set such that $\left(\lambda_\Omega^m,\frac{1}{\nu(\Omega)}\1_\Omega\right)$ is an $m$-eigenpair, then
\item{ (i)} if $\nu(\Omega)<\frac12 \nu(X)$,
$$\left\{ \begin{array}{lll} \{0 \} = M(\1_\Omega, \lambda) \quad &\hbox{if} \ \ 0 < \lambda < \lambda_\Omega^m, \\ \\ \{ c \1_\Omega \, : \, 0 \leq c \leq 1 \}  \subset M(\1_\Omega, \lambda_\Omega^m) \quad & \
\\ \\ \{ \1_\Omega \}  = M(\1_\Omega, \lambda) \quad &\hbox{if} \ \  \lambda > \lambda_\Omega^m\, ;
 \end{array} \right.$$

\item{ (ii)} if $\nu(\Omega)=\frac12\nu(X)$,
$$\left\{ \begin{array}{lll} \{ c  \, : \, 0 \leq c \leq 1 \} = M(\1_\Omega, \lambda) \quad &\hbox{if} \ \ 0< \lambda < \lambda_\Omega^m, \\ \\ \{c\1_\Omega+d\1_{X\setminus\Omega}:0\le d\le c\le 1\}\subset M(\1_\Omega, \lambda_\Omega^m) \quad & \ \\ \\ \{ \1_\Omega \}  = M(\1_\Omega, \lambda) \quad &\hbox{if} \ \  \lambda > \lambda_\Omega^m\, .
 \end{array} \right.$$
 \end{corollary}
\begin{proof} Note that, since $\left(\lambda_\Omega^m,\frac{1}{\nu(\Omega)}\1_\Omega\right)$ is an $m$-eigenpair, $\nu(\Omega)\le \frac12 \nu(X)$.
Now, if $\nu(\Omega)<\frac12 \nu(X)$, then $\hbox{med}_\nu(\1_\Omega)=\{0\}$, and, if
$\nu(\Omega)=\frac12 \nu(X)$, then $\hbox{med}_\nu(\1_\Omega)=\{c\,:\,0\le c\le 1\}$.  Consequently, the result  follows by  Proposition~\ref{proper002} and   Corollary~\ref{continv} with $T(r)=\nu(\Omega)r$.
In the case that $\nu(\Omega)=\frac12 \nu(X)$ we have $$\{ c \1_\Omega \, : \, 0 \leq c \leq 1 \}\cup \{ c  \, : \, 0 \leq c \leq 1 \}\subset M(\1_\Omega, \lambda_\Omega^m),$$ hence,  since  $M(\1_\Omega, \lambda_\Omega^m)$ is convex, we get that
$$\{c\1_\Omega+d\1_{X\setminus\Omega}:0\le d\le c\le 1\}\subset M(\1_\Omega, \lambda_\Omega^m).$$
\end{proof}

\begin{proposition}\label{camp004}
 Let $\Omega$ be a $\nu$-measurable set with $0<\nu(\Omega)<\nu(X)$. If $\1_\Omega\in M(\1_\Omega,\lambda_{\Omega}^m)$ then $\left(\lambda_\Omega^m,\frac{1}{\nu(\Omega)}\1_\Omega\right)$ is an $m$-eigenpair.
\end{proposition}
\begin{proof}
 Let us first see that $\Omega$ is  $m$-calibrable. Indeed, for a $\nu$-measurable subset $E$ of $\Omega$ with $0<\nu(E)<\nu(\Omega)$, we have that
$$\begin{array}{l}P_m(\Omega)=\mathcal{E}_m(\1_\Omega, \1_\Omega, \lambda_\Omega^m)\le \mathcal{E}_m(\1_E, \1_\Omega, \lambda_\Omega^m)\\ \\ =P_m(E)+\lambda_\Omega^m\big(\nu(\Omega)-\nu(E)\big)=
P_m(E)+P_m(\Omega)-\lambda_\Omega^m\nu(E),
\end{array}$$
from where the $m$-calibrability of $\Omega$ follows.

Since $\Omega$ is $m$-calibrable, there exists (see \cite[Theorem 5.8]{MST1}) an antisymmetric function $\g_0$ in $\Omega\times\Omega$ such that
    \begin{equation}
  -1\le \g_0(x,y)\le 1 \qquad \hbox{for $(\nu \otimes m_x)$-\mbox{a.e. }$(x,y) \in \Omega \times \Omega$},\end{equation}    and
\begin{equation}\label{calibrablecondinomega}\lambda_\Omega^m = -\int_{\Omega }\g_0(x,y)\,dm_x(y) + 1 - m_x(\Omega), \quad x \in \Omega.
  \end{equation}
Now, if $\1_\Omega\in M(\1_\Omega,\lambda_{\Omega}^m)$, there exists $\xi_1 \in  \hbox{sign}(0)$ such that   $$\lambda_{\Omega}^m \xi_1 \in \Delta_1^m (\1_\Omega).$$
Therefore, there exists $g_1\in L^\infty(X\times X, \nu \otimes m_x)$ antisymmetric with $\Vert g_1 \Vert_\infty \leq 1$ such that
\begin{equation}\label{minimizerconditionomega}
-\int_{X}\g_1(x,y)\,dm_x(y)=  -\lambda_{\Omega}^m \xi_1(x) \quad \hbox{for }\nu-\mbox{a.e }x\in X,
\end{equation}
and
         \begin{equation}g_1(x,y) \in {\rm sign}(\1_\Omega(y) - \1_\Omega(x)) \quad \hbox{for }(\nu \otimes m_x)-a.e. \ (x,y) \in X \times X.
     \end{equation}
Let
$$g(x,y):=\left\{\begin{array}{cc}
    g_0(x,y) & \hbox{if } (x,y)\in\Omega\times\Omega \\[8pt]
    g_1(x,y) & \hbox{elsewhere}
  \end{array}\right.$$
and
$$\xi(x):=\left\{\begin{array}{cc}
    1 & \hbox{if } x\in\Omega \\[8pt]
    -\xi_1(x) & \hbox{elsewhere.}
  \end{array}\right.$$
  Then, \eqref{calibrablecondinomega} and \eqref{minimizerconditionomega} read as follows
$$\lambda_\Omega^m\xi(x) = -\int_{X }\g(x,y)\,dm_x(y), \quad \hbox{for } \nu-\hbox{a.e. }x \in \Omega$$
and
$$\lambda_\Omega^m\xi(x) = -\int_{X }\g(x,y)\,dm_x(y), \quad \hbox{for } \nu-\mbox{a.e }x\in X\setminus\Omega,$$
thus $(\lambda_\Omega^m,\frac{1}{\nu(\Omega)}\1_\Omega)$ is an $m$-eigenpair.
\end{proof}

\begin{corollary}\label{camp005}
Let $\Omega\subset X$ be a $\nu$-measurable set with $0<\nu(\Omega)< \nu(X)$. The following   statements are equivalent:
\item{ (i)} $\1_\Omega\in M(\1_\Omega,\lambda_\Omega^m)$,
\item{ (ii)} $\left(\lambda_\Omega^m,\frac{1}{\nu(\Omega)}\1_\Omega\right)$ is an $m$-eigenpair,
\item { (iii)} the following thresholding property holds $$
\left\{\begin{array}{ll}
\displaystyle 0\in M(\1_\Omega,\lambda)\quad\forall\,0<\lambda\le \lambda_\Omega^m, \hbox{ and}\\ \\
\displaystyle \1_\Omega\in M(\1_\Omega,\lambda)\quad\forall\,\lambda\ge\lambda_\Omega^m,
\end{array}\right.$$
\end{corollary}

\begin{proof}
$(i) \rightarrow (ii)$ by Proposition~\ref{camp004}. $(ii) \rightarrow (iii)$ by Corollary~\ref{proper003}. $(iii) \rightarrow (i)$ follows easily.
 \end{proof}

 The following result is proved in \cite[Proposition 3.1]{MST1}.
\begin{proposition}[\cite{MST1}]\label{NForm1prop} Let $1\le p\le \infty$. For $u \in BV_m(X) \cap L^{p'}(X,\nu)$, we have that
\begin{equation}\label{NForm1}
TV_m(f) =   \sup \left\{ \int_{X} f(x) ({\rm div}_m \z)(x) d\nu(x)  \ : \ \z \in X_m^p(X), \ \Vert \z \Vert_\infty \leq 1 \right\}.
\end{equation}
\end{proposition}

We say that a function $f \in BV_m(X)\cap L^{p'}(X,\nu)$ is {\it maximal} if the supremum  in \eqref{NForm1} is a maximum, that is, if there exists $\z_0=\z_0(f) \in X_m^p(X)$ with  $\Vert \z_0 \Vert_\infty \leq 1$ such that
\begin{equation}\label{tres001}TV_m(f) = \int_{X} f(x) ({\rm div}_m \z_0)(x) d\nu(x).
\end{equation}

\begin{proposition}\label{casi1} If $f \in BV_m(X)$ is a maximal function with $\z_0=\z_0(f)$ satisfying~\eqref{tres001}, then, for $\lambda_*=\Vert div_mz_0\Vert_\infty$, $$f\in M(f,\lambda_*),$$
and, consequently,  $M(f,\lambda) = \{ f \}$ for all $\lambda >\lambda_*$.
\end{proposition}
\begin{proof} Given $u \in L^1(X,\nu)$, by Proposition \ref{NForm1prop}, we have that
$$\mathcal{E}_m(u, f, \lambda_*)= TV_m(u) + \lambda_* \int_X \vert u - f \vert d\nu \geq \int_{X} u(x) ({\rm div}_m \z_0)(x) d\nu(x)+ \lambda_* \int_X \vert u - f \vert d\nu$$ $$= \int_{X} f(x) ({\rm div}_m \z_0)(x) d\nu(x)+ \lambda_* \int_X \vert u - f \vert d\nu+ \int_{X} (u(x) -f(x)) ({\rm div}_m \z_0)(x) d\nu(x)$$ $$\geq \mathcal{E}_m(f, f, \lambda_*)+ \left( \lambda_* -\Vert {\rm div}_m \z_0 \Vert_\infty \right)\int_X \vert u - f \vert d\nu=\mathcal{E}_m(f, f, \lambda_*).$$
Therefore, $f\in M(f,\lambda_*)$. The rest of the thesis follows from Lemma~\ref{proper001}. \end{proof}

\begin{proposition}\label{pers002}   For any $\nu$-measurable set $\Omega\subset X$, $\1_\Omega$ is a maximal function  with $\z_0=\z_0(\1_\Omega)$ given by
 $$z_0(x,y)=\left\{\begin{array}{ll}
                  0 & \hbox{if }(x,y)\in\Omega\times\Omega, \\[8pt]
                  -1 & \hbox{if }(x,y)\in (X\setminus\Omega)\times\Omega, \\[8pt]
                  1 & \hbox{if }(x,y)\in \Omega\times(X\setminus\Omega), \\[8pt]
                  0 & \hbox{if }(x,y)\in (X\setminus\Omega)\times(X\setminus\Omega) .
                \end{array}\right.$$
Hence   $$\1_\Omega\in M(\1_\Omega,\lambda_*),$$
 where
$$\lambda_*=\lambda_*(\Omega):=\Vert \1_\Omega -m_{(.)}(\Omega) \Vert_\infty$$
satisfies
$0<\lambda_*\le 1.$
\end{proposition}

 \begin{proof}
Indeed, for $x\in\Omega$,
$$({\rm div}_m \z_0)(x)=\frac{1}{2}\left(\int_X z_0(x,y)dm_x(y)-\int_X z_0(y,x)dm_x(y)\right)=$$
$$=\frac{1}{2}\int_{X\setminus \Omega}1 dm_x(y)-\frac12\int_{X\setminus\Omega}-1dm_x(y)=m_x(X\setminus\Omega) . $$
Therefore,
$$\int_{X} \1_\Omega(x) ({\rm div}_m \z_0)(x) d\nu(x)=\int_\Omega\int_{X\setminus\Omega}dm_x(y)d\nu(x)=P_m(\Omega) . $$
Observe also that, for $x\in X\setminus\Omega$,
$$({\rm div}_m \z_0)(x)=-m_x(\Omega) . $$
Therefore,
$${\rm div}_m \z_0 (x)= \1_\Omega(x) -m_{x}(\Omega),$$
and, consequently,
$$\lambda_*=  \Vert {\rm div}_m \z_0 \Vert_\infty=\Vert \1_\Omega -m_{(.)}(\Omega) \Vert_\infty.$$
\end{proof}

\begin{remark} \label{pers001}{\rm
  (i)
    We have that $$\lambda_\Omega^m\le \lambda_*(\Omega).$$ Otherwise, if $\lambda_*(\Omega)< \lambda_\Omega^m$, since $\1_\Omega\in M(\1_\Omega,\lambda_*)$, by Lemma~\ref{proper001} (i) we would have that $\1_\Omega\in M(\1_\Omega,\lambda_\Omega^m)$. Hence, by Proposition~\ref{camp004}, $\left(\lambda_\Omega^m,\frac{1}{\nu(\Omega)}\1_\Omega\right)$ is an $m$-eigenpair and then, by Proposition~\ref{proper003}, $\1_\Omega\notin M(\1_\Omega,\lambda_*)$ which is a contradiction.

Note that, by Proposition~\ref{camp004},
\begin{center}
       if  $\lambda_\Omega^m = \lambda_*(\Omega)$ then $\left(\lambda_\Omega^m,\frac{1}{\nu(\Omega)}\1_\Omega\right)$ is an $m$-eigenpair.
\end{center}

  We point out that in \cite[Theorem 6.5]{MST1}, assuming that $\Omega$ is $m$-calibrable, we proved that $\left(\lambda_\Omega^m,\frac{1}{\nu(\Omega)}\1_\Omega\right)$ is an $m$-eigenpair under the weaker assumption that $\lambda_\Omega^m\ge m_x(\Omega)$ for all $x \in X \setminus \Omega$.

\item{ (ii)} Furthermore, $$\lambda_*(X\setminus\Omega)=\lambda_*(\Omega),$$ and, consequently, from the previous point,
$$ \max\{\lambda_\Omega^m,\lambda_{X\setminus\Omega}^m\}\le \lambda_*(\Omega).$$

} \end{remark}

\begin{proposition}\label{camp006}
Let $\Omega\subset X$ be a $\nu$-measurable set. There exists $\lambda(\Omega)$ satisfying $$ \hbox{max}\{\lambda_\Omega^m,\lambda_{X\setminus\Omega}^m\} \le \lambda(\Omega)\le \lambda_*(\Omega)$$ and
$$\left\{ \begin{array}{ll}
\1_\Omega\not\in M(\1_\Omega,\lambda) & \hbox{if } \ 0<\lambda< \lambda(\Omega),\\[6pt]
         \1_\Omega\in M(\1_\Omega,\lambda(\Omega)), &  \\[6pt]
         \{\1_\Omega\}= M(\1_\Omega,\lambda) & \hbox{if } \ \lambda > \lambda(\Omega).
       \end{array}\right.
$$
Furthermore,
\begin{equation}\label{firstm001}\lambda(\Omega)=\lambda_\Omega^m\ \hbox{ if, and only if, } \left(\lambda_\Omega^m,\frac{1}{\nu(\Omega)}\1_\Omega\right)  \hbox{ is an $m$-eigenpair,}
\end{equation}
 and
\begin{equation}\label{firstm002}\lambda(\Omega)=\lambda_{X\setminus\Omega}^m \ \hbox{ if, and only if, } \left(\lambda_{X\setminus\Omega}^m,\frac{1}{\nu(X\setminus\Omega)}\1_{X\setminus\Omega}\right) \hbox{ is an $m$-eigenpair}.\end{equation}
\end{proposition}

\begin{proof}
By Proposition~\ref{pers002},  $\lambda_*(\Omega) \in \{\lambda \, : \, \1_\Omega\in M(\1_\Omega,\lambda)\}\neq \emptyset$.  Set $$\lambda(\Omega):=\inf\{\lambda \, : \, \1_\Omega\in M(\1_\Omega,\lambda)\}.$$  Then,
 $$  \lambda(\Omega)\le \lambda_*(\Omega),$$
 and, by Proposition~\ref{sec002},
$$\lambda(\Omega)=\min\{\lambda \, : \, \1_\Omega\in M(\1_\Omega,\lambda)\}.$$
Hence, $$\1_\Omega\in M(\1_\Omega,\lambda(\Omega)),$$
 and, by Lemma \ref{proper001}, $\{\1_\Omega\}= M(\1_\Omega,\lambda)$
for every $\lambda>\lambda(\Omega)$.

For $\lambda<\lambda_\Omega^m$, we have
$$\mathcal{E}_m(0, \1_\Omega, \lambda)=\lambda \nu(\Omega)<P_m(\Omega)=\mathcal{E}_m(\1_\Omega, \1_\Omega, \lambda) $$
so $\1_\Omega\not\in M(\1_\Omega,\lambda)$.
Moreover, for $\lambda<\lambda_{X\setminus\Omega}^m$, we have
$$\mathcal{E}_m(\1_X, \1_\Omega, \lambda)=\lambda \nu(X\setminus\Omega)<P_m(X\setminus\Omega)=P_m(\Omega)=\mathcal{E}_m(\1_\Omega, \1_\Omega, \lambda) $$
so $\1_\Omega\not\in M(\1_\Omega,\lambda)$.
Consequently, we have   that
$$ \hbox{max}\{\lambda_\Omega^m,\lambda_{X\setminus\Omega}^m\} \le \lambda(\Omega)\le \lambda_*(\Omega) \, .$$

  Finally, \eqref{firstm001} follows from Corollary~\ref{camp005}, and \eqref{firstm002} follows from  Corollary~\ref{camp005} and Lemma~\ref{sec001}.
\end{proof}

\begin{proposition}\label{laseg01}
  Let $\Omega\subset X$ be a $\nu$-measurable set.
\item{(i) } If there exists $\lambda>0$ such that $0 \in M(\1_\Omega, \lambda)$, then there exists $\lambda^{0}(\Omega)$ satisfying $$ 0 < \lambda^0(\Omega)\le h_1^m(\Omega)$$
    and
    $$\left\{ \begin{array}{ll}
         \hbox{med}_\nu(\1_\Omega)=M(\1_\Omega,\lambda) & \hbox{if } \ 0<\lambda<\lambda^0(\Omega),\\[6pt]
         0\in M(\1_\Omega,\lambda^0(\Omega)), &  \\[6pt]
         0 \not\in M(\1_\Omega,\lambda) & \hbox{if } \ \lambda > \lambda^0(\Omega).
       \end{array}\right.
$$
\item{(ii) } If there exists $\lambda>0$ such that $1 \in M(\1_{\Omega}, \lambda)$, then there exists $\lambda^{1}(\Omega)$ satisfying $$ 0 < \lambda^1(\Omega)\le h_1^m(X\setminus\Omega)$$  and
    $$\left\{ \begin{array}{ll}
         \hbox{med}_\nu(\1_\Omega)=M(\1_\Omega,\lambda) & \hbox{if } \ 0<\lambda<\lambda^1(\Omega),\\[6pt]
         1\in M(\1_\Omega,\lambda^1(\Omega)), &  \\[6pt]
         1 \not\in M(\1_\Omega,\lambda) & \hbox{if } \ \lambda > \lambda^1(\Omega).
       \end{array}\right.$$
\end{proposition}

\begin{proof}
  (i): Let $\tilde{\Omega}\subset \Omega $ be a $\nu$-measurable set, then
$$\mathcal{E}_m(\1_{\tilde{\Omega}}, \1_\Omega, \lambda)-\mathcal{E}_m(0,\1_\Omega,\lambda)=P_m(\tilde{\Omega})-\lambda \nu(\tilde{\Omega}),$$
so that
$$\mathcal{E}_m(\1_{\tilde{\Omega}}, \1_\Omega, \lambda)<\mathcal{E}_m(0,\1_\Omega,\lambda) \iff \lambda>\lambda_{\tilde{\Omega}}^m,$$
 thus  $$\hbox{$0 \in M(\1_\Omega, \lambda)$ implies   $\lambda \leq h_1^m(\Omega)$.}$$
  Therefore, if we set
 $$\lambda^0(\Omega):= \sup \{ \lambda \, : \, 0 \in M(\1_\Omega, \lambda)\},$$
 we have that $\lambda^0(\Omega) \leq h_1^m(\Omega)$. Moreover, by Proposition~\ref{sec002}, we have that $$\lambda^0(\Omega)= \max \{ \lambda>0 \, : \, 0 \in M(\1_\Omega, \lambda)\}$$
and this is the parameter that we were looking for.
\item{ (ii)}   follows from (i) and Lemma~\ref{sec001}.  \end{proof}

  We can set $\lambda^0(\Omega)=0$ if there is no  $\lambda>0$ such that $0 \in M(\1_\Omega, \lambda)$, and $\lambda^1(\Omega)=0$ if there is no  $\lambda>0$ such that $1 \in M(\1_\Omega, \lambda)$.

We have the following formula for the thresholding parameter $\lambda(\Omega)$.
\begin{proposition}\label{jusa001}  Let $\Omega\subset X$ be a $\nu$-measurable set, then
 \label{az001}
\begin{equation}\label{ram002} \lambda(\Omega)=   \sup \left\{ \frac{P_m(\Omega)-P_m(E)}{\nu(\Omega\triangle E)} \, :\, E\subset X \hbox{ \rm $\nu$-measurable, }  \nu(\Omega\triangle E)>0 \right\}.
\end{equation}
\end{proposition}
\begin{proof}
Set
$\alpha:=\sup \left\{ \frac{P_m(\Omega)-P_m(E)}{\nu(\Omega\triangle E)} \, :\, E\subset X \hbox{  $\nu$-measurable, }  \nu(\Omega\triangle E)>0 \right\}$ and let $E\subset\Omega$ be a $\nu$-measurable set with  $\nu(\Omega\triangle E)>0$. Then, since
  $\1_\Omega\in M(\1_\Omega,\lambda(\Omega))$, we have that
  $$\mathcal{E}_m(\1_{E}, \1_\Omega, \lambda(\Omega))=P_m(E)+\lambda(\Omega)\nu(\Omega\triangle E)\ge
  \mathcal{E}_m(\1_\Omega,\1_\Omega,\lambda(\Omega))=
  P_m(\Omega), $$
  from where we obtain that
  $$\lambda(\Omega)\ge\frac{P_m(\Omega)-P_m(E)}{\nu(\Omega\triangle E)},$$
  and, hence, $$\lambda(\Omega)\ge \alpha.$$
 On the other hand, by \eqref{esetformulationN} and the definition of $\alpha$, for every $u \in L^1(X,\nu)$ we have
$$\mathcal{E}_m(u, \1_\Omega, \alpha) = \int_{0}^{1}  \Big(P_m (E_t(u)) + \alpha \nu( E_t(u) \bigtriangleup \Omega)\Big) dt\ge P_m(\Omega)=\mathcal{E}_m(\1_\Omega, \1_\Omega, \alpha),
$$
thus $\1_\Omega\in M(\1_\Omega,\alpha),$ and, consequently,
$$\lambda(\Omega)\le \alpha.$$
\end{proof}

We have the following formula for the thresholding parameter $\lambda^0(\Omega)$.

\begin{proposition}\label{jusa002}  Let $\Omega\subset X$ be a $\nu$-measurable set, then
$$\lambda^0(\Omega)=\inf\left\{ \frac{P_m(E)}{\nu(\Omega)-\nu(\Omega\triangle E)} \ : E\subset X \hbox{ \rm $\nu$-measurable, } \nu(\Omega\triangle E)<\nu(\Omega) \right\}.$$
\end{proposition}
\begin{proof}Set $\alpha:=\inf\left\{ \frac{P_m(E)}{\nu(\Omega)-\nu(\Omega\triangle E)} \ : E\subset X \hbox{  $\nu$-measurable, } \nu(\Omega\triangle E)<\nu(\Omega) \right\}.$ Since
$$\mathcal{E}_m(\1_E,\1_\Omega,\lambda^0(\Omega))=P_m(E)+\lambda^0(\Omega)\nu(\Omega\triangle E) \ \ \hbox{and} \ \ \mathcal{E}_m(0,\1_\Omega,\lambda^0(\Omega))=\lambda^0(\Omega)\nu(\Omega) ,$$
we have that $\lambda^0(\Omega)\le \alpha$.
 Let us see the opposite inequality. For this it is enough to prove that $0\in M(\1_\Omega,\alpha)$, that is
$$\mathcal{E}_m(u, \1_\Omega, \alpha) \ge \mathcal{E}_m(0, \1_\Omega, \alpha) \quad \forall u \in L^1(X,\nu).$$
By \eqref{esetformulationN}, this inequality is  equivalent to
$$\int_{0}^{1}  \Big(P_m (E_t(u)) + \alpha \Big(
\nu( E_t(u) \bigtriangleup \Omega)-\nu(\Omega)\Big)\Big) dt\ge 0 \quad \forall u \in L^1(X,\nu).$$
Now,
$$\begin{array}{c}\displaystyle
 \int_{0}^{1}  \Big(P_m (E_t(u)) + \alpha \Big(
\nu( E_t(u) \bigtriangleup \Omega)-\nu(\Omega)\Big)\Big) dt
\\ \\
\displaystyle
= \int_{\{t:\nu( E_t(u) \bigtriangleup \Omega)-\nu(\Omega)\ge 0\}}  \Big(P_m (E_t(u)) + \alpha \Big(
\nu( E_t(u) \bigtriangleup \Omega)-\nu(\Omega)\Big)\Big) dt\\ \\
\displaystyle
\qquad + \int_{\{t:\nu( E_t(u) \bigtriangleup \Omega)-\nu(\Omega)< 0\}}  \Big(P_m (E_t(u)) + \alpha \Big(
\nu( E_t(u) \bigtriangleup \Omega)-\nu(\Omega)\Big)\Big) dt,
\end{array}$$
but the first integral in the right hand side is trivially non-negative and the second one is also non-negative by the definition of $\alpha$.
\end{proof}

 \begin{remark}\label{remarkthres}{\rm
 Note that, if $\1_E\in M(\1_\Omega,\lambda)$, then
 \begin{equation}\label{toq001}\lambda^-(E)\le \lambda\le\lambda^+(E),
 \end{equation}
  where
$$\lambda^-(E):=\sup \left\{ \frac{P_m(U)-P_m(E)}{\nu(\Omega\triangle E)-\nu(\Omega\triangle U)} :  U\subset X \hbox{  $\nu$-measurable, } \nu(\Omega\triangle U)>\nu(\Omega\triangle E) \right\},$$
$$\lambda^+(E):=\inf \left\{\frac{P_m(U)-P_m(E)}{\nu(\Omega\triangle E)-\nu(\Omega\triangle U)} :  U\subset X \hbox{ $\nu$-measurable, } \nu(\Omega\triangle U)<\nu(\Omega\triangle E) \right\}.$$
Indeed, if $\1_E\in M(\1_\Omega,\lambda)$, then, for any $\nu$-measurable set $U\subset X$,
$$P_m(E)+\lambda \nu(\Omega\triangle E)\le P_m(U)+\lambda \nu(\Omega\triangle U)$$
thus, if $\nu(\Omega\triangle U)>\nu(\Omega\triangle E)$, we have that
$$\lambda\ge \frac{P_m(U)-P_m(E)}{\nu(\Omega\triangle E)-\nu(\Omega\triangle U)}  ,$$
and, if $\nu(\Omega\triangle U)<\nu(\Omega\triangle E)$, we have that
$$\lambda\le \frac{P_m(U)-P_m(E)}{\nu(\Omega\triangle E)-\nu(\Omega\triangle U)}.$$

Furthermore, observe that, if $\1_E\in M(\1_\Omega,\lambda)$, then
  \begin{equation}\label{toq002}P_m(E)=\inf\left\{P_m(U)\,:\, U\subset X \hbox{  $\nu$-measurable, } \nu(\Omega\triangle U)=\nu(\Omega\triangle E)\right\}.
\end{equation}
Conversely, \eqref{toq001} and~\eqref{toq002} imply that $\1_E\in M(\1_\Omega,\lambda)$.}
 \end{remark}

 It is known (see~\cite{DAG}) that a thresholding property for a set in $\mathbb{R}^2$ implies calibrability of the set. From the previous results we obtain  the non-local counterpart of this result.

 \begin{proposition}\label{puvi001}
   Let $\Omega\subset X$ be a $\nu$-measurable set with $0<\nu(\Omega)<\nu(X)$, if there exists a thresholding parameter $\lambda^*>0$ such that
\item{ (1)} $0\in M(\1_\Omega,\lambda)\quad\forall\, 0<\lambda<\lambda^*$, and
\item{ (2)}  $\1_\Omega\in M(\1_\Omega,\lambda)\quad\forall\, \lambda>\lambda^*$,

    \noindent then $$  \lambda(\Omega)=\lambda^*=\lambda_\Omega^m,$$
    and $\left(\lambda_\Omega^m,\frac{1}{\nu(\Omega)}\1_\Omega\right)$ is an $m$-eigenpair. In particular, $\Omega$ is $m$-calibrable.
\end{proposition}

\begin{proof}
    By (1), we have that
    $$\mathcal{E}_m(0, \1_\Omega, \lambda)\le \mathcal{E}_m(\1_\Omega, \1_\Omega, \lambda)  \quad\forall\, 0<\lambda<\lambda^*,$$
    that is,
    $$\lambda\nu(\Omega)\le P_m(\Omega) \quad\forall\, 0<\lambda<\lambda^* ,$$
    from where it follows that
     \begin{equation}\label{camp002} \lambda\le\lambda_\Omega^m \quad\forall\, 0<\lambda<\lambda^* .
    \end{equation}
       Hence, $$ \lambda^*\le \lambda_\Omega^m.$$
       On the other hand, by (2) and the definition of $\lambda(\Omega)$,
       $$\lambda(\Omega)\le \lambda^*.$$
       Then, since $\lambda_\Omega^m\le\lambda(\Omega)$, we get $$\lambda_\Omega^m\le\lambda(\Omega)\le\lambda^*\le\lambda_\Omega^m.$$
       Thus, by Proposition \ref{camp006}, we have that $\left(\lambda_\Omega^m,\frac{1}{\nu(\Omega)}\1_\Omega\right)$ is an $m$-eigenpair.
\end{proof}

The following example proves that the minimizer when the observed image is the characteristic function of a set $\Omega$ need not be the characteristic function of a set contained in $\Omega$. Note that in the continuous setting, when $\Omega$ is convex, it is known that for almost all $\lambda>0$ there is a unique minimizer which, moreover, is the characteristic function of a set contained in $\Omega$ (see \cite[Corollary~5.3]{ChanEsedoglu}). We also observe how, with the ROF model with $L^1$- fidelity term, the scale space is mostly constant and makes sudden transitions at certain values of the scale paramenter. In particular, we see how a set may suddenly vanish.

\begin{example}\label{elej001}{\rm
Consider the locally finite weighted discrete graph with vertex set $X=\{1,2,3,4,5,6\}$ and weights $w_{1,2}=5$, $w_{2,3}=6$, $w_{3,4}=2$, $w_{4,5}=1$ and $w_{5,6}=3$. Let $\Omega=\{1,2\}$.

We have that
$$\left\{ \begin{array}{ll}
              \{0\} = M(\1_{\{1,2\}},\lambda)  & {\rm for } \ 0<\lambda< \frac15=\lambda^0(\Omega), \\[8pt]
            \{c\1_{\{1,2,3,4\}} \, : \, c\in[0,1] \}= M(\1_{\{1,2\}},\lambda) & {\rm for } \ \lambda= \frac15, \\[8pt]
            \{\1_{\{1,2,3,4\}}\}= M(\1_{\{1,2\}},\lambda) & {\rm for } \ \frac15< \lambda< \frac13,   \\[8pt]
            \{\1_{\{1,2,3\}}+c\1_{\{4\}}\, : \, c\in[0,1] \}= M(\1_{\{1,2\}},\lambda) & {\rm for } \ \lambda= \frac13, \\[8pt]
            \{\1_{\{1,2,3\}}\}= M(\1_{\{1,2\}},\lambda) & {\rm for } \ \frac13< \lambda< \frac12,  \\[8pt]
            \{\1_{\{1,2\}}+c\1_{\{3\}}\, : \, c\in[0,1] \}= M(\1_{\{1,2\}},\lambda) & {\rm for } \ \lambda=\frac12, \\[8pt]
            \{\1_{\{1,2\}}\}= M(\1_{\{1,2\}},\lambda) & {\rm for } \ \lambda> \frac12=\lambda(\Omega).
                      \end{array}\right.$$
Indeed, to start with, note that
$$\mathcal{E}_m(\1_{\Omega},\1_\Omega,\lambda)=6=:h_1(\lambda)   ,$$
$$\mathcal{E}_m(\1_{\{1,2,3\}},\1_\Omega,\lambda)=2+8\lambda=:h_2(\lambda)   , $$
$$\mathcal{E}_m(\1_{\{1,2,3, 4\}},\1_\Omega,\lambda)=1+11\lambda=:h_3(\lambda), $$
and
$$  \mathcal{E}_m(0,\1_\Omega,\lambda)=16\lambda=:h_4(\lambda) . $$
 We have that,
\item{ $\bullet$} if $0\le\lambda<\frac15$, then $h_4(\lambda)<h_i(\lambda)$ for $i=1,2,3$,
\item{ $\bullet$} if $\frac15<\lambda<\frac13$, then $h_3(\lambda)<h_i(\lambda)$ for $i=1,2,4$,
\item{ $\bullet$} if $\frac13<\lambda<\frac12$, then $h_2(\lambda)<h_i(\lambda)$ for $i=1,3,4$,
\item{ $\bullet$} and, if $\frac12<\lambda\le1$, then $h_1(\lambda)<h_i(\lambda)$ for $i=2,3,4$.
\\[8pt] Moreover, for any other set $F\subseteq\{1,2,3,4,5,6\}$ different to $\{1,2\},\ \{1,2,3\}$ and $\ \{1,2,3,4\}$, and for any $\lambda>0$, we have that $\mathcal{E}_m(\1_F,\1_\Omega,\lambda)$ is larger than $\min \{h_1(\lambda),h_2(\lambda),h_3(\lambda),h_4(\lambda)\}$.

Following Remark \ref{aboverem01}, to see that $\1_\Omega\in M(\1_\Omega,\frac12)$, take
$$g(1,2)=-\frac{1}{10} , \ g(2,3)=-1 , \ g(3,4)=-1 , \ g(4,5)=-\frac12 , \ g(5,6)=0 $$
and
$$\xi(1)=-\frac15, \ \xi(2)=-1 , \ \xi(3)=1 , \ \xi(4)=1 , \ \xi(5)=\frac14 , \ \xi(6)=0 \ . $$
 For $\lambda<\frac12$, since $h_4(\lambda)>h_3(\lambda)$, we have that
  $\1_\Omega\not\in M(\1_\Omega,\lambda)$. Moreover, by Lemma~\ref{proper001}~(i)  we get that $$\{\1_\Omega\}= M(\1_\Omega,\lambda)\quad\hbox{for }\lambda>\frac12.$$

Since $\mathcal{E}_m(\1_{\{1,2,3\}},\1_\Omega,\frac12)=\mathcal{E}_m(\1_\Omega,\1_\Omega,\frac12) $ we have that $\1_{\{1,2,3\}}\in M(\1_\Omega,\frac12)$ and using the convexity of $M(f,\lambda)$ we get that
$$\big\{\1_{\{1,2\}}+c\1_{\{3\}}\, : \, c\in[0,1] \big\}\subset M\left(\1_{\{1,2\}},1/2\right).$$
 Now, $\{1,2\}$ and $\{1,2,3\}$ are the unique minimizers of $\mathcal{E}_m^G(\cdot, \Omega, 1/2)$, thus, by Theorem~\ref{geomequiv}, we have that
$$\big\{\1_{\{1,2\}}+c\1_{\{3\}}\, : \, c\in[0,1] \big\}= M\left(\1_{\{1,2\}},1/2\right).$$

To see that $\1_{\{1,2,3\}}\in M(\1_\Omega,\frac13)$, take
$$g(1,2)=-\frac15 , \ g(2,3)=-\frac79 , \ g(3,4)=-1 , \ g(4,5)=-1 , \ g(5,6)=0 $$
and
$$\xi(1)=-\frac35, \ \xi(2)=-1 , \ \xi(3)=1 , \ \xi(4)=1 , \ \xi(5)=\frac34 , \ \xi(6)=0 \ . $$
Consequently, by Lemma~\ref{proper001}~(iii) we have that $\1_{\{1,2,3\}}\in M(\1_\Omega,\lambda)$ for  $\frac13\le \lambda\le \frac12$.  Moreover,  $\{1,2,3\}$ is the unique minimizer of $\mathcal{E}_m^G(\cdot, \Omega, \lambda)$ for such $\lambda$'s thus, by Theorem~\ref{geomequiv},
$\1_{\{1,2,3\}}$ is the unique element in  $M(\1_\Omega,\lambda)$ for $\frac13< \lambda< \frac12$.

Since $\mathcal{E}_m(\1_{\{1,2,3,4\}},\1_\Omega,\frac13)=\mathcal{E}_m(\1_{\{1,2,3\}},\1_\Omega,\frac13) $ we have that $\1_{\{1,2,3,4\}}\in M(\1_\Omega,\frac13)$ and, as above,   by Theorem~\ref{geomequiv},
$$\big\{\1_{\{1,2,3\}}+c\1_{\{4\}}\, : \, c\in[0,1] \big\}= M\left(\1_{\{1,2\}},1/3\right).$$

Now, to see that $\1_{\{1,2,3,4\}}\in M(\1_\Omega,\frac15)$, take
$$g(1,2)=-\frac15 , \ g(2,3)=-\frac{8}{15} , \ g(3,4)=-\frac45 , \ g(4,5)=-1 , \ g(5,6)=-\frac{1}{15} $$
and
$$\xi(1)=-1, \ \xi(2)=-1 , \ \xi(3)=1 , \ \xi(4)=1 , \ \xi(5)=1 , \ \xi(6)=\frac13 \ . $$
Then again, by Lemma \ref{proper001} (iii), we have that $\1_{\{1,2,3,4\}}\in M(\1_\Omega,\lambda)$ for $\frac15\le \lambda\le \frac13$ and as before, by Theorem~\ref{geomequiv}, $\left\{\1_{\{1,2,3,4\}}\right\}= M(\1_\Omega,\lambda)$ for $\frac15< \lambda< \frac13$.

Finally, the fact that $\mathcal{E}_m(\1_{\{1,2,3,4\}},\1_\Omega,\frac15)=\mathcal{E}_m(0,\1_\Omega,\frac15) $
gives, by  Theorem~\ref{geomequiv}, that
$$\big\{c\1_{\{1,2,3,4\}} \, : \, c\in[0,1] \big\}= M\left(\1_{\{1,2\}},1/5\right)$$
and, by Lemma \ref{proper001} (ii),   $\{0\}= M(\1_\Omega,\lambda)$ for $0\le \lambda<\frac15$.

Note that $\lambda_\Omega^m=\frac38<\frac12=\lambda(\Omega)$ thus, by Proposition \ref{camp006}, $(\frac38, \frac{1}{16} \1_\Omega)$ is not an $m$-eigenpair. However, $\Omega$ is $m$-calibrable since it consists of two points. Note also that
$$\frac{P_m(\Omega)-P_m(\{1,2,3\})}{\nu(\Omega\triangle \{1,2,3\})}=\frac{6-2}{8}=\frac12=\lambda(\Omega) , $$
$$\frac{P_m(\{1,2,3,4\})}{\nu(\Omega)-\nu(\Omega\triangle \{1,2,3,4\})}=\frac{1}{16-11}=\frac15=\lambda^0(\Omega) , $$
and, regarding Corollary \ref{remarkthres},
$$\frac{P_m(\{1,2,3\})-P_m(\{1,2,3,4\})}{\nu(\Omega\triangle \{1,2,3,4\})-\nu(\Omega\triangle \{1,2,3\})}=\frac{2-1}{11-8}=\frac13 \ . $$
 Finally, observe that by Corollary \ref{mmeancurvatureandlambdaforgraphs},
since
$$H^m_{\partial \{1,2,3,4\}}(4)=1-2m_4(\{1,2,3,4\}=1-2\frac{2}{3}=-\frac13,$$
in order for $  \{1,2,3,4\}$ to be a minimizer of  $\mathcal{E}_m^G(\cdot, \{1,2\}, \lambda)$, we must have
$$\lambda\le \frac13,$$
and $\frac13$ is precisely the upper thresholding parameter for this set.

Observe that, if we add a loop at vertex $4$, $w_{4,4}=\alpha>0$, the set $\{1,2,3,4\}$ can be  a minimizer of  $\mathcal{E}_m^G(\cdot, \{1,2\}, \lambda)$ only if (by~\eqref{gives01})
$$\lambda\le -H^m_{\partial \{1,2,3,4\}}(4)-\frac{w_{4,4}}{d_4}=\frac{1}{3+\alpha}<\frac13\,.$$
}\end{example}

 In the following example, for which we avoid to give as much detail as in the previous one, we can see how, as the value of $\lambda$ is decreased, minimizers become coarser as smaller objects merge together to form larger ones.

\begin{example}\label{elej002}{\rm
In $\mathbb{Z}^2$ with the Hamming distance and $w_{i,j}=1$ for every $i,j$, consider the set $\Omega$ given in Figure \ref{fig1}. Then, for $\frac13<\lambda<\frac25$, the minimizer for the ROF problem with the $L^1$-fidelity term and datum $\1_\Omega$ is the characteristic function of the set $E$ represented in \ref{fig2}. This set $E$ merges together the two components  of $\Omega$. Note that
$$\mathcal{E}_m(\1_{\Omega},\1_\Omega,\lambda)=28   ,$$
$$\mathcal{E}_m(\1_E,\1_\Omega,\lambda)=20+20\lambda   , $$
and
$$  \mathcal{E}_m(0,\1_\Omega,\lambda)=80\lambda . $$
By restricting the ambient space to a big enough bounded subset of $\Z^2$ and recalling Example \ref{JJ}(\ref{dom00606}) we obtain a finite invariant measure and the same calculations work.
\begin{figure}[h]
\centering
  \begin{subfigure}[t]{0.45\textwidth}
  \includegraphics[width=\textwidth]{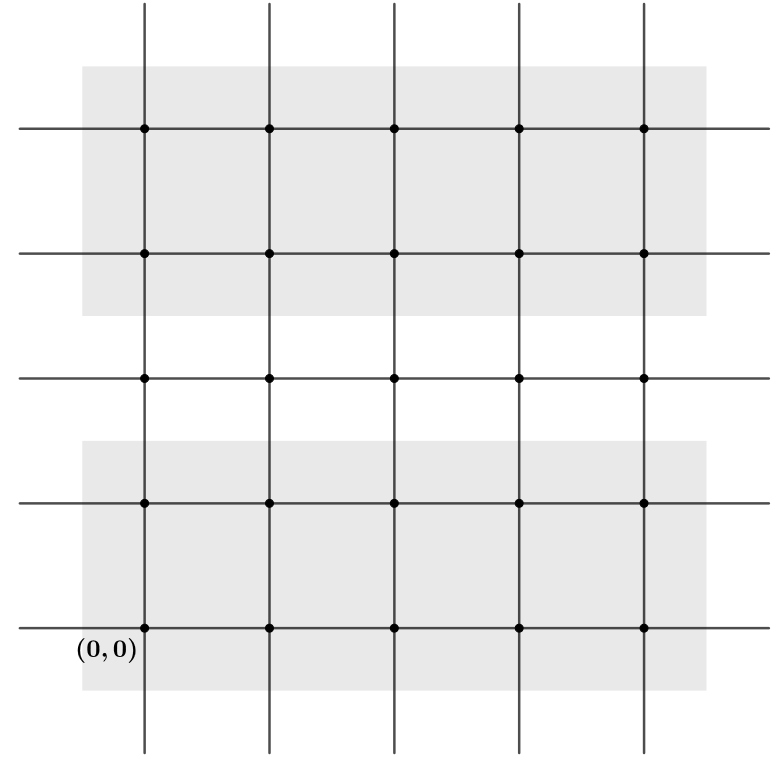}
  \caption{$\Omega$ is the set formed by the points in the shaded region. }
  \label{fig1}
  \end{subfigure}\hspace{1cm}
 \begin{subfigure}[t]{0.45\textwidth}
  \includegraphics[width=\textwidth]{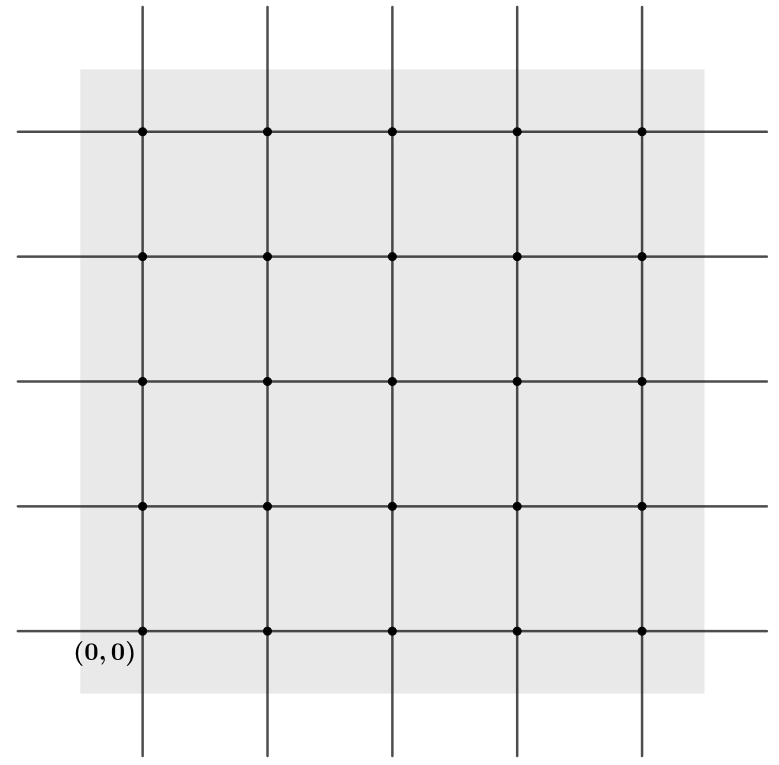}
  \caption{The minimizer, $E$, for $\frac13<\lambda<\frac25$ is the set formed by the points in the shaded region.}
  \label{fig2}
  \end{subfigure}
  \caption{The point $(0,0)$ is labelled in the graphs, and the adjacent points are represented by the dots. }\label{fig}
\end{figure}
}
\end{example}

\subsection{The Gradient Descent Method}
 In order to apply   this method   one needs to solve the Cauchy problem
\begin{equation}\label{L1ParabROFmodN}
 \left\{ \begin{array}{ll} v_t \in \Delta_1^m v(t)- \lambda {\rm sign}(v(t) - f) \quad &\hbox{in} \ (0, T) \times X \\ \\   v(0, x) = v_0(x) \quad &\hbox{in} \ x \in X,\end{array} \right.
 \end{equation}
that  can be rewritten as the following abstract Cauchy problem in $L^2(X,\nu)$
\begin{equation}\label{L1ACP2}
v'(t) + \partial  \mathcal{E}_m(u, f, \lambda)(v(t)) \ni 0, \quad v(0) = v_0.
\end{equation}
 Let  $f$ be in $L^1(X,\nu)$    .  Since $\mathcal{E}_m(\cdot, f, \lambda)$ is convex and lower semi-continuous,  by the theory of maximal monotone operators~(\cite{Brezis}), we have that, for any initial data $v_0 \in L^2(X,\nu)$, problem \eqref{L1ACP2} has a unique strong solution. Therefore, if we define a solution of problem  \eqref{L1ParabROFmodN} as a function $v \in C(0, T; L^2(X,\nu)) \cap W^{1,1}_{loc}(0, T; L^2(X,\nu))$ such that $v(0, x) = v_0(x)$ for $\nu$-a.e. $x \in X$ and such that there exists $\xi(t) \in {\rm sign}(v(t) - f)$ satisfying
$$ \lambda \xi(t) + v_t(t) \in \Delta_1^m (v(t)) \quad \hbox{for almost all} \ \ t \in (0, T),$$
we have the following existence and uniqueness result.

\begin{theorem}\label{L1ExistUniqROF}  For every $v_0 \in L^2( X,\nu)$ there exists a unique strong solution of the Cauchy problem~\eqref{L1ParabROFmodN} in $(0,T)$ for any $T > 0$.
Moreover, we have the following contraction principle in any $L^q(X,\nu)$--space, $1\le q\le \infty$:
\begin{equation}\label{L1contrprin}\Vert v(t)-w(t)\Vert_q\le \Vert v_0-w_0\Vert_q\quad \forall\, 0<t<T,\end{equation}
for any pair of solutions $v,\, w$ of problem~\eqref{L1ParabROFmodN} with initial datum $v_0$, $w_0$ respectively.
\end{theorem}
Note that the contraction principle \eqref{L1contrprin} in any $L^q$-space follows from the fact that the operator $\partial \mathcal{E}_m(\cdot, f, \lambda) $ is completely accretive. Indeed, given $(u_1,v_1), (u_2,v_2) \in \partial \mathcal{E}_m(\cdot, f, \lambda) $ and  $$p \in P_0:= \{ q \in C^\infty (\R) \ : \ 0\leq q' \leq 1, \  \ \hbox{supp}(q') \ \hbox{compact and } \ 0 \not\in \hbox{supp} \},$$
we need to prove that
$$\int_X (v_2 - v_1)p(u_2 - u_1) d\nu \geq 0.$$
Now, there exist $\xi_i \in {\rm sign}(u_i - f)$ such that $v_i - \lambda \xi_i =w_i \in \partial \mathcal{F}_m(u_i)$, $i = 1,2$. Then, since $\partial \mathcal{F}_m$ is a completely accretive operator and
 $$\lambda \int_X (\xi_2 - \xi_1)p(u_2 - u_1) d\nu = \lambda \int_X (\xi_2 - \xi_1)p((u_2 - f) - (u_1- f)) d\nu \geq 0,$$
we have that
$$\int_X (v_2 - v_1)p(u_2 - u_1) d\nu = \int_X (w_2 - w_1)p(u_2 - u_1) + \lambda \int_X (\xi_2 - \xi_1)p(u_2 - u_1) d\nu \geq 0.$$

Let $(T_\lambda(t))_{t \geq 0}$ be the semigroup in $L^2(X, \nu)$ associated with the operator $\partial \mathcal{E}_m(\cdot, f, \lambda) $, that is, $ T_\lambda(t)v_0$ is the solution of problem \eqref{L1ParabROFmodN}.
On  account of the contraction principle we have that for any  $u^*\in M(f,\lambda)$, if $\mathcal{L}_{u^*}(u):= \Vert u - u^* \Vert_2$, then
\begin{equation}\label{Fliapu} \mathcal{L}_{u^*} \ \hbox{is a Lyapunov functional for the semigroup} \ (T_\lambda(t))_{t \geq 0}.
\end{equation}
Indeed, for $t>s$, we have
$$
\mathcal{L}_{u^*}(T_{\lambda}(t)v_0) = \Vert T_{\lambda}(t)v_0-u^*\Vert_2= \Vert T_{\lambda}(t-s)\left(T_{\lambda}(s)v_0\right)-T_{\lambda}(t-s)u^* \Vert_2 $$ $$\le \Vert T_{\lambda}(s)v_0- u^*\Vert_2 = \mathcal{L}_{u^*}(T_{\lambda}(s)v_0).
 $$
 \begin{theorem}\label{Asimpb} Assume that $f \in L^1X, \nu)$. Let $v_0 \in L^2(X, \nu)$ and $v(t):= T_{\lambda}(t)v_0$. If the $\omega$-limit set $$\omega(v_0):= \{ w \in L^2(X, \nu) \ : \ \exists t_n \to + \infty \ s.t. \lim_{n \to \infty} v(t_n) = w \}$$
is non-empty,  then there exists  $u^*\in M(f,\lambda)$ such that
$$\lim_{t \to \infty} v(t) = u^* \quad \hbox{in} \ \ L^2(X, \nu).$$
 \end{theorem}
 \begin{proof} Let  $u^* \in \omega(v_0)$, then there exists $t_n \to + \infty$ such that $$\lim_{n \to \infty} v(t_n) = u^*.$$
 By \cite[Theorem 3.2]{Brezis}, we get
 \begin{equation}\label{e11}
 \frac{d^+ v}{dt}(t) + \partial \mathcal{E}_m(\cdot, f, \lambda)(v(t)) \ni 0 \quad \hbox{for all} \ t\in (0, +\infty)
 \end{equation}
 and, by Theorem~\ref{thin001}, we have that $M(f, \lambda) \not= \emptyset$ thus $0\in R(\partial \mathcal{E}_m(\cdot, f, \lambda))$. Therefore, by \cite[Theorem 3.10]{Brezis},
  \begin{equation}\label{e111}
 \lim_{n \to \infty}\frac{d^+ v}{dt}(t_n) =0.
 \end{equation}
 Since $\partial \mathcal{E}_m(\cdot, f, \lambda)$ is closed, from \eqref{e11} and \eqref{e111} we get
 $$0 \in \partial \mathcal{E}_m(\cdot, f, \lambda)(u^*),$$
 i.e. $u^*\in M(f,\lambda)$. Now, by \eqref{Fliapu},  $\mathcal{L}_{u^*}$  is a Lyapunov functional for the semigroup $(T_\lambda(t))_{t \geq 0}$, from where it follows that
 $$\lim_{t \to \infty} v(t) = u^* \quad \hbox{in} \ \ L^2(X, \nu).$$
 \end{proof}
 Proving that the $\omega$-limit set $\omega(v_0)$ is non-empty is not an easy task here. For example, one could try to proceed with the usual method of proving that the resolvent is compact, but this requires the use of regularity results which are difficult to obtain in our context due to the non-locality of the problem. Nonetheless, in finite graphs it is trivially true that the $\omega$-limit  set is non-empty. Consequently, we have the following result.

\begin{corollary} Let $[X,d,m]$ be the metric random walk space associated to a connected and weighted discrete graph $G = (V(G), E(G))$. Then, for every $v_0 \in L^2(X, \nu)$ and for $v(t):= T_{\lambda}(t)v_0$,  there exists  $u^*\in M(f,\lambda)$ such that
$$\lim_{t \to \infty} v(t) = u^* \quad \hbox{in} \ \ L^2(V(G), \nu_G).$$

\end{corollary}

\bigskip
\noindent {\bf Acknowledgment.} The authors have been partially supported  by the Spanish MICIU and FEDER, project PGC2018-094775-B-100. The second author was also supported by the Spanish MICIU under Grant BES-2016-079019, which is also supported by the European FSE.

\end{document}